\documentclass[reqno,a4paper]{amsart}
\usepackage[T1]{fontenc}
\usepackage{orcidlink}

\newcommand{\lc}{$\ell$}
\usepackage{latexsym, amssymb, amsmath, mathtools}
\usepackage{enumitem}
\usepackage{amsthm}
\usepackage{tikz-cd} %diagrammi
\usepackage{quiver}
\usepackage{nameref,hyperref}
\usepackage{cleveref}
\usepackage{thmtools}
\usepackage{comment}
\theoremstyle{plain} %stile theorem
\newtheorem{theorem}{Theorem}[section]
\newtheorem{proposition}[theorem]{Proposition}
\theoremstyle{definition}
\newtheorem{definition}[theorem]{Definition}
\newtheorem{remark}[theorem]{Remark}
\newtheorem{example}[theorem]{Example}
\newtheorem{examples}[theorem]{Examples}

\newcommand{\al}{\alpha}

\newcommand{\w}{\rightarrow}

\newcommand{\ra}{\rightarrow}

\newcommand{\cd}{\cdot}
\newcommand{\ttt}{\theta}

\newcommand{\noproof}{\hfil\qed}
\newcommand{\spl}{\operatorname{SplExt_{ss}}}
\newcommand{\splb}{\operatorname{SplExt_{ss}^B}}
\newcommand{\splw}{\operatorname{SplExt_{ss}^W}}
\newcommand{\splg}{\operatorname{SplExt_{ss}^G}}

\newcommand{\id}{\operatorname{id}}

\newcommand{\EAct}{\operatorname{EAct}_{\operatorname{ss}}}

\newcommand{\V}{\mathsf{V}}

\DeclareMathOperator{\Ker}{Ker}

\DeclareMathOperator{\op}{op}
\DeclareMathOperator{\Act}{Act}

\DeclareMathOperator{\SplExt}{SplExt}

\newcommand{\Hp}{\ensuremath{\mathsf{Hoops}}}
\newcommand{\BH}{\ensuremath{\mathsf{BHoops}}}
\newcommand{\WH}{\ensuremath{\mathsf{WHoops}}}
\newcommand{\GH}{\ensuremath{\mathsf{GHoops}}}
\newcommand{\PH}{\ensuremath{\mathsf{PHoops}}}
\newcommand{\BA}{\ensuremath{\mathsf{BLAlg}}}
\newcommand{\MV}{\ensuremath{\mathsf{MVAlg}}}
\newcommand{\GA}{\ensuremath{\mathsf{GAlg}}}
\newcommand{\PA}{\ensuremath{\mathsf{PAlg}}}

\newcommand{\Set}{\ensuremath{\mathsf{Set}}}

\begin{document}\sloppy

\title[On actions and split extensions in varieties of hoops]{On actions and split extensions in varieties of hoops: the case of strong section}

\author[M.~Mancini]{Manuel Mancini~\orcidlink{0000-0003-2142-6193}}
\author[G.~Metere]{Giuseppe Metere~\orcidlink{0000-0003-1839-3626}}
\author[F.~Piazza]{Federica Piazza~\orcidlink{0009-0001-1028-9659}}
%\author[F.~Piazza]{Federica Piazza$^{*}$~\orcidlink{0009-0001-1028-9659}}

\address[M.~Mancini, F.~Piazza]{Dipartimento di Matematica e Informatica, Università degli Studi di Palermo, Via Archirafi 34, 90123 Palermo, Italy.}
\email{manuel.mancini@unipa.it; federica.piazza07@unipa.it}

\address[M.~Mancini]{Institut de Recherche en Mathématique et Physique, Université catholique de Louvain, chemin du cyclotron 2 bte L7.01.02, B--1348 Louvain-la-Neuve, Belgium.}
\email{manuel.mancini@uclouvain.be}

\address[G.~Metere]{Dipartimento di Scienze per gli Alimenti, la Nutrizione e l'Ambiente, Università degli Studi di Milano Statale, Via Celoria 2, 20133 Milano, Italy.}
\email{giuseppe.metere@unimi.it}

\address[F.~Piazza]{Dipartimento di Scienze Matematiche e Informatiche, Scienze Fisiche e Scienze della Terra, Università degli Studi di Messina, Viale Ferdinando Stagno d'Alcontres 31, 98166 Messina, Italy.}
\email{federica.piazza1@studenti.unime.it}

%$^*$Corresponding Author.\\
\thanks{This work is supported by the University of Milan, University of Messina, University of Palermo and by the ``National Group for Algebraic and Geometric Structures and their Applications'' (GNSAGA -- INdAM). The first author is supported by the SDF Sustainability Decision Framework Research Project -- MISE decree of 31/12/2021 (MIMIT Dipartimento per le politiche per le imprese -- Direzione generale per gli incentivi alle imprese) -- CUP:~B79J23000530005, COR:~14019279, Lead Partner:~TD Group Italia Srl, Partner:~University of Palermo, and he is also a postdoctoral researcher of the Fonds de la Recherche Scientifique--FNRS. The second and third authors are supported by the National Recovery and Resilience Plan (NRRP), Mission 4, Component 2, Investment 1.1, Call for tender No.~1409 published on 14/09/2022 by the Italian Ministry of University and Research (MUR), funded by the European Union -- NextGenerationEU -- Project Title Quantum Models for Logic, Computation and Natural Processes (QM4NP) -- CUP:~B53D23030160001 -- Grant Assignment Decree No.~1371 adopted on 01/09/2023 by the Italian Ministry of University and Research (MUR)}


\date{}

\begin{abstract}
The aim of this article is to investigate internal actions and split extensions in the variety of hoops. We provide a characterization of split extensions with strong section in terms of \emph{strong external actions}. Beyond the general setting of hoops, the study is extended to the subvarieties of basic hoops, Wajsberg hoops, Gödel hoops and product hoops. Within the setting of basic hoops and their bounded counterparts, BL-algebras, the double negation yields a significant example of split extension with strong section, thus motivating our approach. A connection between strong external actions of hoops and the semidirect product construction introduced by W.~Rump in the cateogory of L-algebras is established.
\end{abstract}
	
\maketitle
%\tableofcontents
\section*{Introduction}\label{int}

BL-Algebras were introduced by P.~Hájek in~\cite{hajek} as the algebraic semantics of \emph{Basic Logic}, the logic of continuous $t$-norms, and they capture the common fragment of the three most relevant \emph{many-valued} logics, namely \emph{Łukasiewicz logic}, \emph{Gödel logic} and \emph{product logic}. 

A $t$-norm on the unit interval is a map $[0,1]^2 \rightarrow [0,1] \colon (x,y) \mapsto x \cdot y$ such that $([0,1], \cd, 1)$ is a commutative totally ordered monoid. There are three fundamental continuous t-norms: the Łukasiewicz $t$-norm, defined by 
\[
x \cd_{L} y = \max \{x+y-1, 0\},
\] 
the Gödel  $t$-norm, defined by
\[
x \cd_{G} y = \min \{x,y\},
\] 
and the product $t$-norm, defined by
\[
x \cd_{P} y = xy.
\] 
It is a classic result that, up to isomorphism, every continuous $t$-norm behaves locally as one of the three described above.

A $t$-norm induces naturally a \emph{residuation} (or implication) defined by 
\[
x \ra y = \sup\{z \in [0,1] \mid z \cd x \leq y\}.
\] 
The implications associated with the three fundamental continuous $t$-norms are
\[
x \ra_{L} y = \min \{1 - x + y,1\},
\]
\[
x \to_{G} y = 
\begin{cases}
1, & \text{if } x \leq y,\\
y, & \text{otherwise},
\end{cases}
\]
and
\[
x \to_{P} y \;=\;
\begin{cases}
1, & \text{if } x \leq y,\\
\frac{y}{x}, & \text{otherwise}.
\end{cases}
\]
The residuation induced by a continuous $t$-norm satisfies the following conditions:
\begin{align*}
x \ra x &= 1, \\ 
x \ra 1 &= 1, \\
1 \ra x &= x
\end{align*}
and 
\[
x \ra y = 1 \quad \text{and} \quad y \ra x = 1 \quad \Longrightarrow \quad x = y.
\]

In~\cite{hajek}, the author studied in detail the three relevant cases and provides an axiomatization for the corresponding varieties of algebras and logics. The variety $\MV$ generated by $([0,1], \cdot_{L}, \ra_{L}, \max, \min, 0,1)$ defines the class of \emph{Wajsberg algebras} or \emph{MV-algebras} (see~\cite{MV1, Wajsberg}), which form the algebraic semantics of Łukasiewicz infinite-valued logic~\cite{MV2, Luka}. Algebras in the variety $\GA$ generated by $([0,1], \cd_{G}, \ra_{G}, \max, \min, 0,1)$ are called \emph{Gödel algebras}, and they form the equivalent algebraic semantics for \emph{Gödel logic}.
The variety $\PA$ of \emph{product algebras} is generated by $([0,1], \cd_{P}, \ra_{P}, \max, \min, 0,1)$ and its associated propositional calculus is \emph{product logic}~\cite{cignolitorrens, godo}. Finally, P.~Hájek introduced the variety $\BA$ of \emph{BL-algebras}, whose associated propositional calculus is \emph{Basic Logic}. He also conjectured that $\BA$ concides with the variety generated by all algebras $([0, 1], \cd, \ra, \max, \min, 0, 1)$, where $\cd$ is a continuous $t$-norm on the unit interval. This conjecture was then proved in~\cite{cignoliestevagodo}, while a shorter and more algebraic proof can be found in~\cite{basichoop}.

From a categorical point of view, the variety $\BA$, likewise the subvarieties $\MV$, $\GA$ and $\PA$, determines an ideally exact category (see~\cite{IdeallyExact, rel}). Moreover, if $L_2=\{0,1\}$ denotes the two-element Boolean algebra, then the semi-abelian categories $(\BA \downarrow L_{2})$, $(\MV \downarrow L_{2})$, $(\GA \downarrow L_{2})$ and $(\PA \downarrow L_{2})$ are equivalent, respectively, to the varieties of basic hoops, Wajsberg hoops, Gödel hoops, and product hoops.

The algebraic structure now known as a \emph{hoop} was first introduced by B.~Bosbach in~\cite{bosbach1,bosbach2}, where it appeared under the name \emph{complementary semigroups} (\emph{komplementäre Halbgruppen}). The term hoop itself was later coined in an unpublished manuscript by J.~R.~Büchi and T.~M.~Owens~\cite{BuchiOwens}, and has since become standard in the literature on substructural logics and residuated structures.  

A hoop is a commutative, residuated, integral monoid satisfying divisibility. Indeed, hoops can be seen as a \emph{positive} counterpart of residuated lattices, and the subclass of basic hoops can be obtained as the subreducts of the class of BL-algebras by omitting the lattice operations and the constant $0$. The main subclasses of interest are basic hoops, Wajsberg hoops (the hoop subreducts of MV-algebras), Gödel hoops, and product hoops, each of which reflects the equational behavior of its corresponding algebraic variety.

Now, one key notion that can be explored in the context of semi-abelian categories~\cite{Semi-Ab} is that of internal action~\cite{IntAct}, which generalizes classical algebraic notions like group or Lie algebra actions, and provides a description of split extensions (and, therefore, of retractions) in algebraic terms, namely by means of semidirect products. Internal actions operated by the objects of a category were defined by F.~Borceux, G.~Janelidze and G.~M.~Kelly with the aim of extending the correspondence between actions and split extensions from the context of groups and Lie algebras to arbitrary semi-abelian categories. However, in some cases, such as for Orzech categories of interest~\cite{Orzech}, it is more convenient to describe internal actions in terms of \emph{external actions}, i.e., via a set of maps which satisfy a certain set of identities (see for instance~\cite{WRAAlg} and~\cite{CigoliManciniMetere}, where actions in varieties of non-associative algebras are studied).
 
The aim of this article is to investigate internal actions and split extensions in the variety of hoops, as well as in its subvarieties of basic, Wajsberg, Gödel, and product hoops. 

The paper is organized as follows. After this introduction, \Cref{sec_prel1} presents some  background on BL-algebras, hoops, and subvarieties. The notions of split extension, internal action and semidirect product in the context of semi-abelian varieties of Universal Algebra are then recalled in \Cref{sec_prel2}.
In \Cref{sec_strogsec} we focus on split extensions with \emph{strong section}, i.e., such that the corresponding split epimorphism has a strong section in the sense of W.~Rump (see~\cite{Rump1, Rump2}). Such split extensions are described in terms of \emph{strong external actions}, which are pairs of maps satisfying a set of identities closely related to the axioms satisfied by the hoop.
In \Cref{sec_exact} we prove there is a bijection between the set of strong external actions and the set of isomorphism classes of split extensions with strong section in the variety of hoops. The result is then generalized to the subvarieties of basic, Wajsberg, Gödel and product hoops. Furthermore, we prove that this bijection extends to a natural isomorphism between the functor $\EAct(-,X)$ of strong external actions on a fixed object $X$ and the functor $\spl(-,X)$ of isomorphism classes of split extensions with strong section with fixed kernel $X$. We also show that the notion of split extension with strong section trivializes in the context of MV-algebras, whereas in the variety of Gödel hoops, strong external actions coincide with those in the variety of basic hoops.
In addition, in \Cref{sec_rump} we show a connection between the notion of strong external action in the variety of hoops and the semidirect product construction introduced by W.~Rump in the category of L-algebras (see~\cite{Rump3, Rump2}).
We end the manuscript with some possible future directions.

\section{Preliminaries}\label{sec_prel1}
In this section, we present the algebraic structures that are the focus of our study. For background material and further details, we refer the reader to~\cite{hoops}. 

\begin{definition}
A \emph{commutative integral residuated lattice} (or CIRL) is an algebra $A=(A, \vee, \wedge, \cd, \ra, 1)$ of type $(2,2,2,2,0)$ such that
\begin{itemize}
\item[(1)] $(A,\cd, 1)$ is a commutative monoid.
\item[(2)] $(A, \vee, \wedge, 1)$ is a lattice with top element $1$.
\item[(3)] $(\cd, \ra)$ form a \emph{residuated pair} with respect to the lattice ordering, i.e., the following \emph{residuation property} holds:
\[
x \cd y \leq z \quad \text{if and only if} \quad x \leq y \ra z,
\]
where $x,y,z$ denote arbitrary elements of $A$ and $\leq$ is the order induced by the lattice structure.
\end{itemize}
\end{definition}

Residuated lattices, originally introduced in~\cite{res_lattice} as models of the divisibility properties of ideals in commutative rings, have recently acquired importance as the algebraic counterparts of certain fuzzy logics (see, for instance~\cite{esteva, galatos, gerla, hajek}).

We call \emph{bounded} commutative integral residuated lattice (BCIRL, for short) a CIRL with an extra constant $0$ in the signature that is the bottom element in the lattice order~\cite{projectivity, cignoli}.

\begin{definition}\cite{prod}
A \emph{BL-algebra} is an algebra $A=(A, \vee, \wedge, \cd, \ra, 0,1)$ such that
\begin{itemize}
\item[(1)] $(A,\vee,\wedge, \cd, \ra, 0, 1)$ is a BCIRL. \hfill
\item[(2)] $x \wedge y = x \cd (x \ra y)$, for any $x,y \in A$. \hfill (\emph{divisibility condition})
\item[(3)] $(x \ra y) \vee (y \ra x) = 1$, for any $x,y \in A$. \hfill (\emph{prelinearity condition})
\end{itemize}
\end{definition}

We denote by $\BA$ the variety of BL-algebras. In any BL-algebra $A$, one can define the negation of an element $x$ as $\neg x \coloneqq x \ra 0.$ 

We observe that the divisibility condition is equivalent to requiring that the order $\leq$ induced by the lattice operations satisfies the following condition: for any $x,y \in A$
\[
x \leq y \; \text{ if and only if there exists $z \in A$ such that } x= y \cd z.
\]
On the other hand, the prelinearity condition characterizes those BCIRLs generated by chains (i.e., BCIRLs whose lattice order is total). In the present article, we are interested in the following subvarieties of BL-algebras.

\begin{definition}\label{def_alg}
Let $A$ be a BL-algebra. We say that $A$ is a
\begin{itemize}
\item \emph{MV-algebra}~\cite{MV1, turunen} if it satisfies \emph{involutivity}:
\[
\neg \neg x=x,
\]
for any $x \in A$.
\item \emph{Gödel algebra}~\cite{hajek} if it satisfies \emph{idempotency}:
\[
x\cd x=x,
\]
for any $x \in A$.
\item \emph{product algebra}~\cite{hajek} if it satisfies the identity
\[
\neg x \vee ((x \ra x \cd y) \ra y)=1,
\]
for any $x,y \in A$.
\end{itemize}
\end{definition}

We denote by $\MV$, $\GA$ and $\PA$ the varieties of MV-algebras, Gödel algebras and product algebras respectively. 

\begin{remark}\cite{hajek}
BL-algebras, MV-algebras, Gödel algebras and product algebras constitute the equivalent algebraic semantics for \emph{Hájek's Basic Logic}, \emph{infinite-valued Łukasiewicz logic}, \emph{Gödel logic} and \emph{product Logic}.
\end{remark}

\begin{examples}\label{ex}{\ }
\begin{itemize}
\item[1.] The algebra
\[
[0,1]_{\operatorname{MV}} = ([0,1], \cd_L, \ra_L, \max, \min, 0,1),
\]
where $x \cd_L y = \max \{x+y-1,0\}$ and $x \ra_L y = \min \{1-x+y,1\}$, is an MV-algebra, called the \emph{standard MV-algebra}.
\item[2.] The algebra
\[
[0,1]_{G}=([0,1], \cd_{G}, \ra_{G}, \max, \min, 0,1),
\]
where $x \cd_{G}y = \min \{x,y\}$ and
\[
x\ra_{G} y = \begin{cases} 1, \ \ &\text{if} \ x \leq y, \\
y, &\text{otherwise} \end{cases}
\]
is a Gödel algebra, called the \emph{standard Gödel algebra}.
\item[3.] The algebra
\[
[0,1]_P = ([0,1], \cd_P, \ra_P, \max, \min, 0,1),
\]
where $x \cd_P y = xy$ and 
\[
x\ra_P y = \begin{cases} 1, \ \ &\text{if} \ x \leq y, \\
\frac{y}{x}, &\text{otherwise} \end{cases}
\]
is a product algebra, called the \emph{standard product algebra}.
\end{itemize}
\end{examples}

\begin{remark}
As mentioned in the introduction, the algebras of \Cref{ex} generate, respectively, the varieties $\MV$, $\GA$, and $\PA$ (see \cite{MV1, MV2, hajek}).
\end{remark}

In the variety of BL-algebras, and consequently in its subvarieties, the lattice operations of join and meet may be expressed in terms of the binary operations $\cd$ and $\ra$  as follows: 
\begin{align*}
&x \wedge y=x \cd (x \ra y),\\
&x \vee y=((x \ra y)\ra y)\wedge ((y \ra x)\ra x).
\end{align*}
As a result, BL-algebras may be equivalently described within the reduced language of hoops.

\begin{definition}%\cite{hoops}
A \emph{hoop} is an algebra $H=(H, \cdot, \w, 1)$ of type $(2,2,0)$ such that
\begin{enumerate}
\item[H1.] $(H,\cdot, 1)$ is a commutative monoid;
\item[H2.] $x \w x=1$;
\item[H3.] $x \cdot (x \w y)=y \cdot (y \w x)$;
\item[H4.] $(x \cdot y)\w z= x \w (y \w z)$,
\end{enumerate}
for any $x,y,z \in H$. 

A hoop $H$ is \emph{bounded} if there is a constant $0 \in H$ such that
\begin{itemize}
\item[H5.] $0 \w x=1$,
\end{itemize}
for any $x \in H$.
\end{definition}

We denote by $\Hp$ the variety of hoops.
Every hoop $H$ is endowed with a partial order $\leq$, called the \emph{natural order}, which is defined by the following equivalent conditions: for any $x,y \in H$
\begin{enumerate}
\item[(i)] $x \leq y$;
\item[(ii)] $x \w y=1$;
\item[(iii)] there exists $z \in H$ such that $x=z \cdot y$.
\end{enumerate}

In particular, $(H, \leq)$ is a partially ordered set with top element $1$, in which every pair of elements $x,y \in H$ admit an infimum given by $x \wedge y= x \cd (x \ra y)$. If the hoop $H$ is bounded, then $0$ is its bottom element.

We further recall that a \emph{filter} of a hoop $H$ is a subset $F \subseteq H$ such that $(F, \cd, 1)$ is a submonoid of $(H, \cd, 1)$ which is upward closed with respect to the natural order~$\leq$ of $H$. Every filter $F$ of a hoop $H$ defines a congruence $\theta_{F}$ as follows:
\[
\theta_{F} \coloneqq \{(x,y)\in H \times H \mid (x \ra y)\cd (y \ra x) \in F\}.
\]
Conversely, every congruence $\theta$ on a hoop $H$ determines a filter $F_{\theta}=1/\theta$ (the equivalence class of $1$ under $\theta$). The correspondence $F \mapsto \theta_{F}$ establishes an isomorphism between the lattice of filters and the lattice of congruences of $H$, whose inverse is given by the map $\theta \mapsto F_{\theta}$. Therefore, the quotient of $H$ by a filter $F$ can be denoted by $H/F$, where $H/F \coloneqq H/\theta_{F}$. 

A hoop homomorphism $f \colon H \ra H'$ is a homomorphism of monoid which preserves the binary operation $\ra$. If $H$ and $H'$ are bounded, then a homomorphism of bounded hoops must preserve the bottom element $0$. The kernel of a homomorphism $f \colon H \ra H'$ is the set
\[
\Ker f=\{h \in H \mid f(h)=1\}
\]
and it turns out to be a filter of $H$. Conversely, every filter $F$ of a hoop $H$ may be seen as the kernel of the canonical projection map $\pi \colon H \ra H/F$, thus establishing a bijection between kernels and filters.

We now introduce some subvarieties of $\Hp$ which are stictly related to the subvarieties of $\BA$ introduced in \Cref{def_alg}.

\begin{definition}\cite{basichoop, prod} 
Let $H$ be a hoop. We say that $H$ is a
\begin{itemize} 
\item \emph{Basic hoop} if it satisfies the identity
\[
((x \ra y)\ra z)\ra (((y \ra x)\ra z)\ra z)=1,
\]
for any $x,y,z \in H$.
\item \emph{Wajsberg hoop} if it satisfies the identity 
\[
(x \ra y)\ra y=(y \ra x)\ra x,
\]
for any $x,y \in H$.
\item \emph{Gödel hoop} if it is a basic hoop satisfying
\[
x \cd x=x,
\]
for any $x \in X$.
\item \emph{Product hoop} if it is a basic hoop satisfying
\[
(y \ra z)\vee ((y \ra (x \cd y))\ra x)=1,
\]
for any $x,y,z \in H$.
\end{itemize}
\end{definition}

We denote by $\BH$, $\WH, \GH$, and $\PH$ the varieties of basic hoops, Wajsberg hoops, Gödel hoops, and product hoops respectively. We observe that, although it is not not explicitly required by the definition, every Wajsberg hoop is in fact a basic hoop~\cite{hoops}. 

The following proposition makes clear the connection between the subvarieties of hoops introduced above and the subvarieties of BL-algebras of \Cref{def_alg}.

\begin{proposition}\cite{prod}
Basic hoops, Wajsberg hoops, Gödel hoops, and product hoops are the $0$-free subreducts\footnote{Given a BCIRL $A$, a $0$-free subreduct of $A$ is a subalgebra of the $0$-free reduct of $A$~\cite{prod}.} of, respectively, BL-algebras, MV-algebras, Gödel algebras, and product algebras. \noproof
\end{proposition}

It was proved in~\cite{basichoop} that bounded basic hoops are term-equivalent to the variety of BL-algebras. Similarly, bounded Wajsberg hoops, bounded Gödel hoops, and bounded product hoops are term-equivalent, respectively, to the varieties of MV-algebras, Gödel algebras, and product algebras (see~\cite{prodred, aglianò, cignolimundici, productlogic}).

\section{Semidirect products in semi-abelian varieties}\label{sec_prel2}

In this section, we aim to describe the construction of semidirect products in the variety of hoops. Indeed, the class of hoops forms a variety in the sense of Universal Algebra, which has been shown to be a semi-abelian category (in the sense of~\cite{Semi-Ab}) by S.~Lapenta, G.~Metere and L.~Spada in~\cite{rel} using the following characterisation.

\begin{theorem}\cite{caratt}
A variety $\V$ is semi-abelian if and only if it has, among its terms, a unique constant $e$, together with $n$ binary terms $\alpha_{i}$, $i=1,\ldots,n$, and an $(n+1)$-ary term $\theta$ satisfying the identities
\[
\alpha_{i}(x,x)=e, \quad \text{for any} \ i=1,\ldots,n
\]
and
\[
\theta(\underline{\alpha}(x,y),y)=x,
\]
where $\underline{\alpha}(x,y)=(\alpha_{1}(x,y),\ldots,\alpha_{n}(x,y))$. \noproof
\end{theorem}

\begin{remark}\cite{rel}\label{rem_terms}
The variety $\Hp$ is semi-abelian, with $e=1$,
\[
\alpha_{1}(x,y)= x\ra y, \quad \alpha_{2}(x,y)=((x \ra y)\ra y)\ra x 
\]
and
\[
\theta(x,y,z)=(x \ra z)\cd y.
\]
Hence, its subvarieties $\BH, \WH, \GH$ and $\PH$ are semi-abelian as well.
\end{remark}

\begin{remark}
The initial object of the categories $\BA, \MV, \GA$ and $\PA$ is the two-element Boolean algebra $L_{2}=\{0,1\}$. However, these varieties above are not semi-abelian, since they are not pointed, but, as shown in~\cite{rel}, they are protomodular (see~\cite{malcev, caratt, Bourn-Janelidze:Semidirect}). As a consequence,  they are ideally exact categories in the sense of G.~Janelidze~\cite{IdeallyExact}.
\end{remark}

A notion that we need in the following is that of \emph{split extension}. 

\begin{definition}
Let $\V$ be a semi-abelian category, and let $B$ and $X$ be objects of~$\V$. A \emph{split extension} of $B$ by $X$ is a diagram of the form
\begin{equation}\label{eq_spl_ext}
\begin{tikzcd}
X & A & B
\arrow["k", from=1-1, to=1-2]
\arrow["p", shift left, from=1-2, to=1-3]
\arrow["s", shift left, from=1-3, to=1-2]
\end{tikzcd}  
\end{equation}
in $\V$ such that $p \circ s=\id_{B}$ and $(X,k)$ is a kernel of $p$.
\end{definition}
We recall that in any pointed category $\V$, a \emph{kernel} of a morphism $p \colon A \rightarrow B$ is defined as the pullback of $p$ along the unique morphism $0 \rightarrow B$, where $0$ denotes the zero object of $\V$. If $\V$ is a pointed variety, with $0=\{e\}$, then a kernel of a morphism $p \colon A \rightarrow B$ is the subalgebra $p^{-1}(e)$, together with the canonical inclusion. Dually, a \emph{cokernel} of $p$ is defined as the pushout of $p$ along the unique morphism $A \rightarrow 0$.

Notice that protomodularity implies that the pair $(k,s)$ is jointly strongly epic and $p$ is a cokernel of $k$. We recall from~\cite{acc} that a pair of morphisms
\[
\begin{tikzcd}
X & Z & Y
\arrow["f", from=1-1, to=1-2]
\arrow["g", swap, from=1-3, to=1-2]
\end{tikzcd}  
\]
in an arbitrary category~$\V$ is said to be \emph{jointly strongly epic}, or \emph{jointly strongly epimorphic}, when for each commutative diagram:
% https://q.uiver.app/#q=WzAsNSxbMSwwLCJNIl0sWzEsMSwiUCJdLFsxLDIsIloiXSxbMCwyLCJYIl0sWzIsMiwiWSJdLFszLDIsImYiLDJdLFs0LDIsImciXSxbMCwxLCJtIiwyXSxbMiwxLCJcXHBoaSJdLFszLDAsImYnIiwwLHsiY3VydmUiOi0yfV0sWzQsMCwiZyciLDIseyJjdXJ2ZSI6Mn1dLFsyLDAsIlxccHNpIiwyLHsiY3VydmUiOjIsInN0eWxlIjp7ImJvZHkiOnsibmFtZSI6ImRhc2hlZCJ9fX1dXQ==
\[
\begin{tikzcd}
& M & \\
& P \\
X & Z & Y
\arrow["m"', from=1-2, to=2-2]
\arrow["{f'}", curve={height=-12pt}, from=3-1, to=1-2]
\arrow["f"', from=3-1, to=3-2]
\arrow["\psi"', curve={height=12pt}, dashed, from=3-2, to=1-2]
\arrow["\phi", from=3-2, to=2-2]
\arrow["{g'}"', curve={height=12pt}, from=3-3, to=1-2]
\arrow["g", from=3-3, to=3-2]
\end{tikzcd}\]
if $m$ is a monomorphism, then there exists a unique (dashed) morphism $\psi \colon Z \rightarrow M$ such that $m \circ \psi =\phi$. 

A morphism of split extensions from
\begin{equation*}
\begin{tikzcd} 
X' \arrow [r, "k'"] 
&A' \arrow[r, shift left, "p'"] & 
B'\ar[l, shift left, "s'"]
\end{tikzcd}
\end{equation*}
to
\begin{equation*}
\begin{tikzcd} 
X \arrow [r, "k"]
&A \arrow[r, shift left, "p"] & 
B \ar[l, shift left, "s"] 
\end{tikzcd}
\end{equation*}
consists of a triple of morphisms in $\V$ 
\[
(g \colon X' \rightarrow X, \, h \colon A' \rightarrow A, \, f \colon B' \rightarrow B)
\]
such that the following diagram commutes
\begin{equation*}
\begin{tikzcd} 
X' \arrow [r, "k'"] \ar[d, swap, "g"]
&A' \arrow[r, shift left, "p"] \ar[d, swap, "h"]& 
B' \ar[l, shift left, "s'"]\ar[d, "f"]\\
X \arrow [r, "k"]
&A \arrow[r, shift left, "p"] & 
B,\ar[l, shift left, "s"]
\end{tikzcd}
\end{equation*}
i.e., $h \circ k'=k \circ g$, $h \circ s'=s \circ f$ and $p \circ h = f \circ p'$. Let us observe that, again by protomodularity, a morphism of split extensions fixing $X$ and $B$ is necessarily an isomorphism. For  an object $X$ of $\V$, we define the functor
\[
\SplExt(-,X)\colon \V^{\operatorname{op}} \rightarrow \Set
\]
which assigns to any object $B$ of $\V$, the set $\operatorname{SplExt}(B,X)$ of isomorphism classes of split extensions of $B$ by $X$, and to any arrow $f\colon B'\to B$ the \emph{change of base} function  $f^*\colon \operatorname{SplExt}(B,X) \to \operatorname{SplExt}(B',X)$ given by pulling back along $f$. Here, \emph{pulling back along $f$} means that, given a split extension \eqref{eq_spl_ext}, one forms the pullback of $p$ along $f\colon B' \rightarrow B$, obtaining a split extension
% https://q.uiver.app/#q=WzAsNixbMCwwLCJYIl0sWzAsMSwiWCJdLFsxLDAsIkEnIl0sWzIsMCwiQiciXSxbMiwxLCJCIl0sWzEsMSwiQSJdLFswLDEsIiIsMCx7ImxldmVsIjoyLCJzdHlsZSI6eyJoZWFkIjp7Im5hbWUiOiJub25lIn19fV0sWzAsMiwiayciXSxbMiwzLCJwJyIsMCx7Im9mZnNldCI6LTF9XSxbMywyLCJzJyIsMCx7Im9mZnNldCI6LTF9XSxbMyw0LCJmIl0sWzUsNCwicCIsMCx7Im9mZnNldCI6LTF9XSxbNCw1LCJzIiwwLHsib2Zmc2V0IjotMX1dLFsyLDUsIlxcdGhldGEiLDJdLFsxLDUsImsiLDJdXQ==
\[
\begin{tikzcd}
X & {A\times_BB'} & {B'} \\
X & A & B
\arrow["{k'}", from=1-1, to=1-2]
\arrow[equals, from=1-1, to=2-1]
\arrow["{p_2}", shift left, from=1-2, to=1-3]
\arrow["p_1"', from=1-2, to=2-2]
\arrow["{s'}", shift left, from=1-3, to=1-2]
\arrow["f", from=1-3, to=2-3]
\arrow["k"', from=2-1, to=2-2]
\arrow["p", shift left, from=2-2, to=2-3]
\arrow["s", shift left, from=2-3, to=2-2]
\end{tikzcd}
\]
where $p_1$ and $p_2$ are the pullback projections, and $s'=\langle s \circ f, \id_{B'} \rangle$ is the unique section of $p_2$ obtained by the universal property of the pullback. Notice that $p_2$ has the same kernel as $p$, with $k'=\langle k,0_{X,B'}\rangle$, where $0_{X,B'} \colon X \to 0\to B'$ denotes the \emph{zero map}.
%In any pointed category, given two objects $X,B$, one can define the \emph{zero map} $0_{X,B}$, as the composition $X\to 0\to B$.

A feature of semi-abelian categories is that one can define a notion of internal action~\cite{IntAct}. If we fix an object $X$, actions on $X$ give rise to a functor 
\[
\Act(-,X)\colon \V^{\operatorname{op}} \rightarrow \Set.
\]
We will not describe explicitly internal actions here, since they correspond to split extensions via a semidirect product construction (more in detail, there is a natural isomorphism of functors $\operatorname{Act}(-,X)\cong \operatorname{SplExt}(-,X)$, see~\cite{BJK2}), and split extensions are handier to work with. Our aim now is to remind such a construction for semi-abelian varieties of Universal Algebra~\cite{semidir}.

Let $\V$ be a semi-abelian variety and let
\[
\begin{tikzcd}
A & B
\arrow["p", shift left, from=1-1, to=1-2]
\arrow["s", shift left, from=1-2, to=1-1]
\end{tikzcd}
\]
be a split epimorphism in $\V$ such that $X=\Ker{p}$. One can define two set maps
\begin{align*}
&\varphi \colon X^{n}\times B \rightarrow A \colon (\underline{x},b)\mapsto \theta(\underline{x},s(b)), \\
&\psi \colon A \rightarrow X^{n}\times B \colon
a\mapsto (\underline{\alpha}(a, sp(a)), p(a))
\end{align*}
and it is easy to check that $\varphi \circ \psi = \id_{A}$ and $A$ is in bijection with the subset
\[
Y=\{(\underline{x},b)\in X^{n}\times B \mid \underline{\alpha}(\theta(\underline{x}, s(b)),s(b))=\underline{x}\}
\]
of $X^{n}\times B$, where $\underline{\alpha}$ and $\theta$ are as in \Cref{rem_terms}.

\begin{theorem}\cite{semidir}
Given a semi-abelian variety $\V$ and a split epimorphism
\[
\begin{tikzcd}
A & B
\arrow["p", shift left, from=1-1, to=1-2]
\arrow["s", shift left, from=1-2, to=1-1]
\end{tikzcd}\] in $\V$, let $X=\Ker p$ and let $\xi \colon B \flat X \rightarrow X$ be the corresponding internal action of $B$ on $X$. Then the semidirect product $X \rtimes_{\xi} B$ of $X$ and $B$ with respect to $\xi$ is the set
\[
Y=\{(\underline{x},b)\in X^{n}\times B \mid \underline{\alpha}(\theta(\underline{x}, s(b)),s(b))=\underline{x}\}
\]
endowed with the following structure: if $\omega$ is an $m$-ary term of the variety, then in~$Y$ we have 
\begin{align*}\omega_{Y}((\underline{x}_{1},b_{1}),\ldots,&(\underline{x}_{m},b_{m}))\\
=&(\underline{\alpha}_{A}(\omega_{A}(\theta_{A}(\underline{x}_{1},s(b_{1})),\ldots,\theta_{A}(\underline{x}_{m},s(b_{m}))),\omega_{A}(s^{m}(\underline{b}))),\omega_{B}(\underline{b})). \noproof
\end{align*}
\end{theorem}

\subsection{Semidirect products in the variety of hoops}
Since the variety $\Hp$ is semi-abelian, we can specialize the results of the previous section to this context. Let
\[
\begin{tikzcd}
A & B
\arrow["p", shift left, from=1-1, to=1-2]
\arrow["s", shift left, from=1-2, to=1-1]
\end{tikzcd}
\]
be a split epimorphism in $\mathsf{Hoops}$ and let $X=\Ker p$. One can define two set maps
\begin{align*}
&\varphi \colon X^{2}\times B \rightarrow A \colon
(x,x',b)\mapsto ((x\rightarrow s(b))\cdot x'), \\
&\psi \colon A\rightarrow X^{2}\times B \colon
a \mapsto ( a\rightarrow sp(a), ((a \ra sp(a))\ra sp(a))\ra a, p(a))
\end{align*}
and, again, one has $\varphi \circ \psi =\id_{A}$ and $A$ is in bijection with the set
\begin{align*}  
Y=\{(x,x',b)\in X^{2}&\times B \mid ((x\rightarrow s(b))\cdot x')\rightarrow s(b)=x, \\ & ((((x \ra s(b))\cd x')\ra s(b))\ra s(b))\ra ((x \ra s(b))\cd x')\}.
\end{align*}
Indeed, we have
\begin{align*}\alpha_{1}(\theta(x,x',s(b)),s(b))&=\alpha_{1}(((x \ra s(b))\cd x'), s(b))\\
&=((x \ra s(b))\cd x')\ra s(b)
\end{align*}
and
\begin{align*}
\al_{2}(\ttt(x,x',s(b)),s(b))&=\al_{2}(((x \ra s(b))\cd x'),s(b))\\
&=((((x \ra s(b))\cd x' )\ra s(b))\ra s(b))\ra ((x \ra s(b))\cd x').
\end{align*}

Thus, we can state the following.
\begin{theorem}
Let
\[
\begin{tikzcd}
A & B
\arrow["p", shift left, from=1-1, to=1-2]
\arrow["s", shift left, from=1-2, to=1-1]
\end{tikzcd}\]
be a split epimorphism in $\mathsf{Hoops}$, let $X=\Ker p$ and let $\xi \colon B \flat X\rightarrow X$ the corresponding internal action of $B$ on $X$. Then the semidirect product $X \rtimes_{\xi} B$ is the set
\begin{align*}  
Y=\{(x,x',b)\in X^{2}&\times B \mid ((x\rightarrow s(b))\cdot x')\rightarrow s(b)=x, \\ & ((((x \ra s(b))\cd x')\ra s(b))\ra s(b))\ra ((x \ra s(b))\cd x')=x'\}.\end{align*}
endowed with the following operations
\begin{align*}
(x,x'&,b)\rightarrow (y,y',b')=((((x\rightarrow s(b))\cdot x')\rightarrow ((y\rightarrow s(b'))\cdot y'))\rightarrow (s(b\rightarrow b')), \\ &((((((x\ra s(b))\cd x')\ra ((y\ra s(b'))\cd y'))\ra s(b \ra b'))\\ &\ra s(b \ra b'))\ra (((x\ra s(b))\cd x')\ra ((y\ra s(b'))\cd y'))),\ b\rightarrow b'),\end{align*}
\begin{align*}
(x,x'&,b)\cdot (y,y',b')=((((x\rightarrow s(b))\cdot x')\cd ((y\rightarrow s(b'))\cdot y'))\rightarrow (s(b\cd b')), \\ &((((((x\ra s(b))\cd x')\cd ((y\ra s(b'))\cd y'))\ra s(b \cd b'))\\ &\ra s(b \cd b'))\ra (((x\ra s(b))\cd x')\cd ((y\ra s(b'))\cd y'))),\ b\cd b')\end{align*}
and unit $1_{X \rtimes_{\xi} B}=(1,1,1)$. \noproof
\end{theorem}

%\begin{proof}
%For any binary operation $* \in \{\ra, \cd \}$, we have
%\begin{align*}(x,x',b)*(y,y',b')&=(\underline{\alpha}_{A}(\theta_{A}(x,x',s(b))*\theta_{A}(y,y',s(b')),s(b)*s(b'))
%\\&
%=((((x\rightarrow s(b))\cdot x')* ((y\rightarrow s(b'))\cdot y'))\rightarrow (s(b * b')), \\ &((((((x\ra s(b))\cd x')* ((y\ra s(b'))\cd y'))\ra s(b * b'))\\ &\ra s(b * b'))\ra (((x\ra s(b))\cd x')* ((y\ra s(b'))\cd y'))),\ b*b').
%\end{align*}
%\end{proof}

\section{Split extensions with strong section}\label{sec_strogsec}
The equations that define the semidirect product and the operations in the variety $\Hp$ are easier to handle when the split extension has strong section in the sense of W.~Rump (see~\cite{Rump1,Rump2}).

\begin{definition}
A split extension 
\[
\begin{tikzcd}
X & A & B
\arrow["k", from=1-1, to=1-2]
\arrow["p", shift left, from=1-2, to=1-3]
\arrow["s", shift left, from=1-3, to=1-2]
\end{tikzcd}\]
in the variety $\Hp$ has \emph{strong section} if the morphism $s$ is a strong section of $p$, i.e., if the equation 
\begin{equation}\label{ss}
a \ra s(b)=sp(a)\ra s(b)
\end{equation} holds for any $a \in A$ and $b \in B$.
\end{definition}

\begin{remark}
If $a \in X$, then \Cref{ss} can be written as
\[
a \ra s(b)=s(b),
\]
since $(X,k)$ is a kernel of $p$. Here, we identify $a$ with $k(a)$.
\end{remark}

\begin{example}\label{ex_direct}
Any direct product gives rise to a split extension with strong section. Let $X$ and $B$ be hoops and consider their direct product $A=X \times B$, 
endowed with the componentwise operations
\begin{align*}
(x,b) \cd (y,b')&= (x \cd y, b \cd b'),\\
(x,b) \ra (y,b')&=(x \ra y, b \ra b')    
\end{align*}
and unit $(1,1)$. Then, it is defined a split extension
\[
\begin{tikzcd}
X & {X \times B} & B
\arrow["\iota_1", from=1-1, to=1-2]
\arrow["\pi_2", shift left, from=1-2, to=1-3]
\arrow["\iota_2", shift left, from=1-3, to=1-2]
\end{tikzcd}
\]
where $\iota_1(x)=(x,1)$, $\iota_2(b)=(1,b)$ and $\pi_2(x,b)=b$, and one may check that $\iota_2$ is a strong section of $p_2$. Indeed, for any $(x,b) \in X \times B$ and $b' \in B$, one has
\begin{align*}
(x,b) \ra \iota_2(b')
&= (x,b) \ra (1,b') \\
&= (x \ra 1,\, b \ra b') \\
&= (1,\, b \ra b') \\
&= (1,b) \ra (1,b') \\
&= \iota_2\pi_2(x,b) \ra \iota_2(b').
\end{align*}
\end{example}

We observe that the property of having a strong section is invariant under isomorphism of split extensions. Thus, for any hoop $X$, we can define the functor
\[
\operatorname{SplExt}_{\operatorname{ss}}(-,X) \colon \Hp^{\operatorname{op}} \ra \Set
\]
which assigns to any hoop $B$, the set $\SplExt_{\operatorname{ss}}(B,X)$ of isomorphism classes of split extensions with strong section of $B$ by $X$ in $\Hp$, and to every hoop morphism $f \colon B' \ra B$ the change of base function $f^{*} \colon \SplExt_{\operatorname{ss}}(B,X)\ra \SplExt_{\operatorname{ss}}(B',X)$ given by pulling back along $f$.

\begin{proposition}\label{prop_ss}
If the split extension 
\[
\begin{tikzcd}
X & A & B
\arrow["k", from=1-1, to=1-2]
\arrow["p", shift left, from=1-2, to=1-3]
\arrow["s", shift left, from=1-3, to=1-2]
\end{tikzcd}
\]
has strong section, then $A$ is in bijection with the set
\[
Y=\{(x,x',b)\in X^{2}\times B \mid x=1 \emph{ and } s(b)\ra (s(b)\cd x')=x'\}.
\]
\end{proposition}

\begin{proof}
One has
\[
((x\ra s(b))\cd x')\ra s(b)=(s(b)\cd x')\ra s(b)=1
\]
and
\begin{align*}
((((x \ra s(b))\cd x' )&\ra s(b))\ra s(b))\ra ((x \ra s(b))\cd x')\\
&=(((s(b)\cd x')\ra s(b))\ra s(b))\ra (s(b)\cd x')\\
&=((x'\ra (s(b)\ra s(b)))	\ra s(b))\ra (s(b)\cd x')\\&=(1 \ra s(b))\ra (s(b)\cd x')\\
&=s(b)\ra (s(b)\cd x')
\end{align*}
for any $x,x' \in X$ and $b \in B$.
\end{proof}

\begin{theorem}\label{thss}
Let
\[
\begin{tikzcd}
X & A & B
\arrow["k", from=1-1, to=1-2]
\arrow["p", shift left, from=1-2, to=1-3]
\arrow["s", shift left, from=1-3, to=1-2]
\end{tikzcd}
\] be a split extension with strong section in $\Hp$ and let $\xi \colon B \flat X \ra X$ be the corresponding internal action of $B$ on $X$. Then the semidirect product $X \rtimes_{\xi} B$ is given by the set
\[
Y'=\{(x,b)\in X \times B \mid s(b)\ra(s(b)\cdot x)=x\}
\]
endowed with the operations
\begin{align*}
(x,b)\ra (y,b')&=(s(b'\ra b)\ra (x \ra y),\ b \ra b'),\\
(x,b)\cdot (y,b')&=(s(b\cdot b') \ra (s(b\cdot b') \cdot x \cdot y),\ b \cdot b')
\end{align*}
and unit $1_{X \rtimes_{\xi} B}=(1,1)$.
\end{theorem}

\begin{proof}
It follows from \Cref{prop_ss} that the underlying set of $X \rtimes_{\xi} B$ is 
\[
\{(x,b)\in X \times B \mid s(b)\ra(s(b)\cdot x)=x\}.
\]
Moreover, one has
\begin{align*}
(x,b)\ra (y,b')&=(\al_{2}(\ttt(1,x,s(b))\ra \ttt(1,y,s(b')),s(b)\ra s(b')), \ b \ra b')\\
&=(\al_{2}((x \cd s(b))\ra (y \cd s(b')), s(b \ra b')), \ b \ra b')\\
&=(((((x \cd s(b))\ra (y \cd s(b')))\ra s(b \ra b'))\\&\ra s(b \ra b'))\ra (((x \cd s(b))\ra ( y \cd s(b')))), \ b \ra b')\\
&=(s(b \ra b')\ra ((x \cd s(b))\ra (y \cd s(b'))), \ b \ra b')\\
&\overset{\text{H4}}{=}((s(b \ra b')\cd (x \cd s(b)))\ra ( y \cd s(b')), \ b \ra b')\\
&=((s(b \cd (b \ra b'))\cd x)\ra (y \cd s(b')), \ b \ra b')\\
&\overset{\text{H3}}{=}((s(b' \cd (b' \ra b))\cd x)\ra (y \cd s(b')), \ b \ra b')\\
&=((s(b' \ra b)\cd s(b')\cd x)\ra (y \cd s(b')), \ b \ra b')\\
&\overset{\text{H4}}{=}(s(b'\ra b)\ra ((s(b')\cd x)\ra (y \cd s(b'))), \ b \ra b')\\
&\overset{\text{H4}}{=}(s(b' \ra b)\ra (x \ra (s(b')\ra (y \cd s(b')))), \ b \ra b')\\
&=(s(b' \ra b)\ra (x \ra y), \ b \ra b')
\end{align*}
where the labels over the equality signs point to the used axioms. In the above computations, we used the identity
\begin{align*}
((x \cd s(b))\ra (y \cd s(b')))\ra s(b \ra b')&=sp((x \cd s(b))\ra (y \cd s(b')))\ra s(b \ra b')\\
&=((1 \cd s(b))\ra (1 \cd s(b')))\ra s(b \ra b')\\
&=(s(b)\ra s(b'))\ra s(b\ra b')=1
\end{align*}
and the fact that $s(b)\ra (s(b)\cd x)=x$, for any $(x,b)\in X\rtimes_{\xi}B$. Furthermore,
\begin{align*}
(x,b)\cd (y,b')&=(\al_{2}((x \cd s(b))\cd (y \cd s(b'),s(b\cd b')), b\cd b')\\
&=((((x\cd y \cd s(b\cd b'))\ra s(b\cd b'))\ra s(b\cd b'))\ra (x\cd y \cd s(b\cd b')))\\
&=(s(b\cd b')\ra (x\cd y \cd s(b\cd b'))),
\end{align*}
since
\begin{align*}
(x \cd y \cd s(b\cd b'))\ra s(b\cd b')&=sp(x \cd y \cd s(b \cd b'))\ra s(b \cd b')\\
&=(1 \cd 1 \cd s(b \cd b'))\ra s(b \cd b')\\
&=s (b \cd b')\ra s(b\cd b')=1
\end{align*}
for any $x,y \in X$ and $b,b' \in B$.
\end{proof}

The following proposition makes clear how to describe the lattice operations of the semidirect product $X\rtimes_{\xi} B$ in the variety of basic hoops.

\begin{proposition}\label{lattice_strongsect}
Let
\[
\begin{tikzcd}
X & A & B
\arrow["k", from=1-1, to=1-2]
\arrow["p", shift left, from=1-2, to=1-3]
\arrow["s", shift left, from=1-3, to=1-2]
\end{tikzcd}
\]
be a split extension with strong section in the variety $\BH$. Then the lattice operations of join and meet of $X \rtimes_{\xi} B$ can be described as
\[
(x,b)\wedge (y,b')=(s(b \wedge b')\ra (s(b \wedge b')\cd (x \wedge y)), \ b \wedge b')
\]
and
\begin{align*}(x,b)\vee (y,b')&=(s(b \vee b') \ra (s(b \vee b')\cd (((s(b'\ra b)\ra (x \ra y))\ra y)\\ &\wedge (s(b\ra b')\ra (y \ra x))\ra x)), \ b \vee b').
\end{align*}
\end{proposition}

\begin{proof}
One has
\begin{align*}
(x,b)\wedge (y,b')&=(x,b)\cd ((x,b)\ra (y,b'))\\
&=(x,b)\cd(s(b'\ra b)\ra (x \ra y), \ b \ra b')\\
&=(s(b \cd (b \ra b'))\ra (s(b \cd (b \ra b'))\cd x \cd (s(b'\ra b)\ra (x \ra y)), \ b \cd (b \ra b'))\\
&=(s(b \wedge b')\ra (s(b')\cd s(b'\ra b)\cd x \cd (s(b'\ra b)\ra (x \ra y)), \ b \wedge b')\\
&\overset{\text{H3}}{=}s(b \wedge b')\ra (s(b')\cd x \cd (x \ra y)\cd ((x \ra y)\ra s(b'\ra b))), \ b \wedge b')\\
&=(s(b \wedge b')\ra (s(b') \cd s(b'\ra b)\cd (x \wedge y), \ b \wedge b')\\
&=(s(b \wedge b') \ra (s(b\wedge b')\cd (x \wedge y)), \ b \wedge b').
\end{align*}
Here we used the identity
\[
(x \ra y)\ra s(b'\ra b)=sp(x \ra y)\ra s(b'\ra b)=1 \ra s(b'\ra b)=s(b'\ra b).
\]
Moreover,
\begin{align*}
(x,b)\vee (y,b')&=(((x,b)\ra (y,b'))\ra (y,b'))\wedge (((y,b')\ra (x,b))\ra (x,b))\\
&=((s(b'\ra b)\ra (x\ra y), \ b \ra b')\ra (y,b'))\wedge ((s(b \ra b')\\ &\ra (y \ra x), \ b'\ra b)\ra (x,b)\\
&=(s(b'\ra (b \ra b'))\ra ((s(b'\ra b)\ra (x\ra y))\ra y), \\ &(b \ra b')\ra b') \wedge (s(b \ra (b' \ra b))\\ &\ra ((s(b \ra b')\ra (y \ra x))\ra x), \ (b' \ra b)\ra b)\\
&=((s (b' \ra b)\ra (x \ra y))\ra y, \ (b \ra b')\ra b')\\ &\wedge((s(b \ra b')\ra (y \ra x)) \ra x, \ (b' \ra b)\ra b)\\
&=(s(b \vee b') \ra (s(b \vee b')\cd (((s(b'\ra b) \\ &\ra (x \ra y))\ra y)\wedge ((s(b\ra b')\ra (y \ra x))\ra x))), \ b \vee b'),
\end{align*}
for any $(x,b),(y,b') \in X \rtimes_\xi B$, since
\begin{align*}
b' \ra (b \ra b')=&(b \cd b')\ra b'\\
=&b \ra (b' \ra b')\\
=&b \ra 1=1\\
=&b \ra (b' \ra b).
\end{align*}
\end{proof}

We give now an example of split extension with strong section in the subvariety of basic hoops.

\begin{example}\cite{Rump1}\label{ex_Rump}
Let $A$ be a BL-algebra. We denote by $\operatorname{MV}(A)$ the set of \emph{regular elements} of $A$
\[
\operatorname{MV}(A)=\{x \in A \mid \neg \neg x=x\},
\]
and by $\operatorname{D}(A)$ the set of \emph{dense} elements of $A$
\[
\operatorname{D}(A)=\{x \in A \mid \neg \neg x=1\}.
\]
By~\cite[Theorem 1.2]{cignolitorrens1}, $\operatorname{MV}(A)$ is a subalgebra of $A$, which is an MV-algebra, and the map $p \colon A \ra \operatorname{MV}(A)$ defined by $p(a)=\neg \neg a$, is a split epimorphism of BL-algebras with kernel $\operatorname{D}(A)$. Thus, if we denote by $i$ and $j$ the canonical inclusions of $\operatorname{MV}(A)$ and $\operatorname{D}(A)$ inside $A$, we have a split extension
\[
\begin{tikzcd}
{\operatorname{D}(A)} & A & {\operatorname{MV}(A)}
\arrow["i", from=1-1, to=1-2]
\arrow["p", shift left, from=1-2, to=1-3]
\arrow["j", shift left, from=1-3, to=1-2]
\end{tikzcd}
\]
in the variety $\BH$ which has strong section, i.e., for any $x \in A$ and $y \in \operatorname{MV}(A)$ we have
\begin{equation}\label{eq_rump}
x \ra y= \neg \neg x \ra y.
\end{equation}
Indeed, let $\{A_{i}\}_{i \in I}$ be an indexed family of BL-chains such that $A$ is the subdirect product of $A_{i}$, for $i \in I$\footnote{The Hájek's subdirect representation theorem (see~\cite[Lemma 2.3.16]{hajek}) states that every BL-algebra $A$ is a subdirect product of totally ordered BL-algebras (BL-chains).}. Then, for any $i \in I$, we have:
\begin{itemize}
\item If $x_{i} \in \operatorname{MV}(A_{i})$, then $\neg \neg x_{i}=x_{i}$ and
\[
x_{i} \ra y_{i}=\neg \neg x_{i} \ra y_{i},
\]
for any $y_i \in \operatorname{MV}(A_i)$.
\item If $x_{i} \in \operatorname{D}(A_{i})$, then
$\neg \neg x_{i}=1_{i}$ and
\[
x_{i}\ra y_{i}=y_{i}=1_i \ra y_i,
\]
for any $y_i \in \operatorname{MV}(A_i)$.
\end{itemize}
Since any BL-chain $A_i$ is the ordinal sum of its regular elements and its dense elements (see~\cite[Theorem 2.2]{Busaniche}), \eqref{eq_rump} holds in $A_{i}$, and hence in $A$.
\end{example}

\section{Strong external actions of hoops}\label{sec_exact}

In this section, we aim to describe split extensions with strong section in the variety $\Hp$ in terms of strong external actions.

\begin{definition}\label{ext_act}
Let $B,X$ be hoops. A \emph{strong external action} of $B$ on $X$ consists of a pair of maps
\begin{align*}
&f \colon B \times X \rightarrow X \colon (b,x) \mapsto f_b(x),\\
&g \colon B \times X \rightarrow X \colon (b,x) \mapsto g_b(x)    
\end{align*}
such that
\begin{enumerate}
%\item[E1.] $x \leq y \  \Rightarrow \ g_{b}(x)\leq g_{b}(y)$;
\item[E1.] $f_{b}(1)=g_{b}(1)=1$; 
\item[E2.] $f_{1}=g_{1}=\id_{X}$;
%\item[E4.] $g_{b_{1} \cdot b_{2}}=g_{b_{1}}\circ g_{b_{2}}$;
%\item[E5.] $g_{b}(x \ra y)=g_b(x) \ra g_{b}(y)$;
\item[E3.] $f_{b_{1} \cdot b_{2}}(x \cdot g_{b_{1}}(x \ra y))=f_{b_{1} \cdot b_{2}}(x \cdot (x \ra y))$;
%\item[E6.] $f_{b_{1} \cdot (b_1 \ra b_{2})}(x \cdot g_{b_{1}}(x \ra y))=f_{b_{1} \cdot (b_1 \ra b_{2})}(x \cdot (x \ra y))$;
%\item[E4.OLD] $g_{(b_{3} \ra (b_{1} \cdot b_{2}))}(f_{b_{1} \cdot b_{2}}(x \cdot y) \ra z) = 
%g_{(b_{2} \ra b_{3}) \ra b_{1}}(x \ra g_{b_{3} \ra b_{2}}(y \ra z))$,
\item[E4.] $g_{(b_{3} \ra (b_{1} \cdot b_{2}))}(f_{b_{1} \cdot b_{2}}(x \cdot y) \ra f_{b_{3}}(z)) = 
g_{(b_{2} \ra b_{3}) \ra b_{1}}(x \ra g_{b_{3} \ra b_{2}}(y \ra f_{b_{3}}(z))$,
\end{enumerate}
for any $b,b_1,b_2,b_3 \in B$ and $x,y,z \in X$. 
\end{definition}
    
We denote by $\EAct(B,X)$ the set of the strong external actions of $B$ on $X$.
    
\begin{remark}
Strong external actions give rise to a functor
\[
\EAct(-,X) \colon \Hp^{\operatorname{op}} \ra \Set
\]
which maps every hoop $B$ to $\EAct(B,X)$, and every morphism $\phi \colon B' \ra B$ in $\Hp$ to the function
\[
\EAct(\phi, X) \colon \EAct(B,X) \ra \EAct(B',X)
\]
which sends an external action $f,g \colon B \times X \rightarrow X$ to the external action $f',g' \colon B' \times X \rightarrow X$ defined by
\[
f'(b',x)=f(\phi(b'),x) \; \text{ and } \; g'(b',x)=g(\phi(b'),x).
\]
\end{remark}

Our goal now is to make clear the connection between strong external actions and split extensions with strong section in the variety of hoops. In fact, the next two propositions show that the two notions are equivalent.

\begin{proposition}\label{split_to_strong}
Let $B,X$ be hoops and let $f,g \colon B \times X \rightarrow X$ be a strong external action. Then the set
\[
Y'=\{(x,b) \in X \times B \mid f_{b}(x)=x\}
\]
endowed with the operations 
\begin{align}
\label{op1}(x,b)\ra (y,b')&=(g_{b'\ra b}(x \ra y), \ b\ra b'),\\
\label{op2}(x,b)\cdot (y,b')&=(f_{b\cdot b'}(x \cdot y), \ b \cdot b')
\end{align}
and unit $1_{Y'}=(1,1)$ is a hoop. Furthermore, the split extension
\begin{equation}\label{splext}
\begin{tikzcd}
X & Y' & B
\arrow["{\iota_{1}}", from=1-1, to=1-2]
\arrow["{\pi_{2}}", shift left, from=1-2, to=1-3]
\arrow["{\iota_{2}}", shift left, from=1-3, to=1-2]
\end{tikzcd}
\end{equation}
where $\iota_{1}(x)=(x,1)$, $\pi_{2}(x,b)=b$ and $\iota_{2}(b)=(1,b)$, has strong section.
\end{proposition}

\begin{proof}
We start by showing that $Y'$ is a hoop. Again, the labels above the equality signs point to the axioms used in the steps.
\begin{itemize}
\item[H1.] It follows from axiom E1 that
\[
(x,b)\ra (x,b)=(g_{b \ra b}(x \ra x), \ b \ra b)=(g_{1}(1), 1)=(1,1),
\]
for any $(x,b) \in Y'$.
\item[H2.] We have
\begin{align*}
(x,b)\cd ((x,b)\ra (y,b'))&=(x,b)\cd(g_{b' \ra b}(x \ra y), \ b \ra b')\\
&=(f_{b \cd (b \ra b')}(x \cd g_{b'\ra b}(x \ra y)), \ b\cd (b \ra b'))\\
&=(f_{b' \cd (b' \ra b)}(x \cd g_{b'\ra b}(x \ra y)), \ b\cd (b \ra b'))\\
&\overset{\text{E3}}{=}(f_{b'\cd (b' \ra b)}(x \cd (x \ra y)), \ b \cd ( b \ra b'))\\
&=(f_{b\cd (b \ra b')}(y \cd (y \ra x)), \ b' \cd (b' \ra b))\\
&\overset{\text{E3}}{=}(f_{b' \cd (b'\ra b)}(y \cd g_{b\ra b'}(y \ra x)), \ b' \cd (b'\ra b))\\
&=(y,b')\cd ((y,b')\ra (x,b)),
\end{align*}
for any $(x,b),(y,b') \in Y'$.
\item[H3.] We have
\begin{align*}
((x,b_1)\cd(y,b_2))\ra (z,b_3)&=(f_{b_{1}\cd b_{2}}(x\cd y), \ b_{1}\cd b_{2})\ra (z,b_{3})\\
&=(g_{b_{3}\ra (b_{1}\cd b_{2})}(f_{b_{1}\cd b_{2}}(x\cd y)\ra z), \ (b_{1}\cd b_{2}) \ra b_{3})
\end{align*}
and
\begin{align*}
(x, b_{1})\ra ((y,b_{2})\ra &(z, b_{3}))=
(x,b_{1})\ra (g_{b_{3}\ra b_{2}}(y \ra z), \ b_{2}\ra b_{3})\\
&=(g_{(b_{2}\ra b_{3})\ra b_{1}}(x \ra g_{b_{3}\ra b_{2}}(y \ra z)), \ b_{1}\ra (b_{2}\ra b_{3})),
\end{align*}
for any $(x,b_1),(y,b_2),(z,b_3) \in Y'$. Hence, by axiom E4 we get
\[
((x,b_{1})\cd (y, b_{2}))\ra (z,b_{3})=(x,b_{1})\ra ((y,b_{2})\ra (z,b_{3})).
\]
\end{itemize}
Thus, $Y'$ is a hoop. Furthermore, the split extensions \eqref{splext} has strong section since
\begin{align*}
(x,b) \ra \iota_2(b')&=(x,b) \ra (1,b')\\
&=(x \ra g_{b \ra b'}(1),b \ra b')\\
&=(x \ra 1, b \ra b')=(1,b \ra b')
\end{align*}
and
\begin{align*}
\iota_2\pi_2(x,b) \ra \iota_2(b')=&\iota_2(b) \ra \iota_2(b')\\
=&(1,b) \ra (1,b')=(1,b \ra b').
\end{align*}
\end{proof}

\begin{proposition}\label{strong_to_split}
Let 
\[\begin{tikzcd}
X & A & B,
\arrow["k", from=1-1, to=1-2]
\arrow["p", shift left, from=1-2, to=1-3]
\arrow["s", shift left, from=1-3, to=1-2]
\end{tikzcd}\] 
be a split extension with strong section in $\Hp$. Then the pair of maps
\[
f,g \colon B \times X \rightarrow X
\]
defined by 
\[
f_{b}(x)=s(b)\ra (s(b)\cdot x) \; \text{ and } \; g_{b}(x)=s(b)\ra x
\]
gives rise to a strong external action of $B$ on $X$.
\end{proposition}

\begin{proof}
We have to show that $(f,g)$ satisfies the axioms of \Cref{ext_act}.
\begin{itemize}
\item[E1.] For any $b \in B$, one has $f_{b}(1)=s(b)\ra (s(b)\cd 1)=s(b)\ra s(b)=1$ and $g_{b}(1)=s(b)\ra 1=1$.
\item[E2.] For any $x \in X$, one has $f_{1}(x)=s(1)\ra (s(1)\cd x)=1 \ra (1 \cd x)=x$ and $g_{1}(x)=s(1)\ra x=1 \ra x=x$.
\item[E3.] We have
\begin{comment}
\begin{align*}
f_{b\cd b'}(x \cd g_{b}(x \ra y))&=s(b\cd b')\ra (s(b\cd b')\cd (x \cd (s(b)\ra(x \ra y))))\\
&=s(b\cd b')\ra (s(b\cd b')\cd (x \cd (x \ra (s(b)\ra y))))\\
&=s(b \cd b')\ra (s(b\cd b')\cd ((s(b)\ra y)\cd ((s(b)\ra y)\ra x)))\\
&=s(b\cd b')\ra (s(b')\cd y \cd (y \ra s(b))\cd ((s(b)\ra y)\ra x))\\
&=s(b\cd b')\ra (s(b\cd b')\cd y\cd ((s(b)\ra y)\ra x))\\
&=s(b\cd b')\ra (s(b\cd b')\cd y \cd ((s(b)\ra y)\ra (s(b)\ra s(b)\cd x)))\\
&=s(b\cd b')\ra (s(b\cd b')\cd y \cd (s(b)\cd (s(b)\ra y)\ra s(b)\cd x))\\
&=s(b\cd b')\ra (s(b\cd b')\cd y \cd (y\cd (y \ra s(b))\ra s(b)\cd x))\\
&=s(b\cd b')\ra (s(b\cd b')\cd y \cd (y\cd s(b)\ra s(b)\cd x))\\
&=s(b\cd b')\ra (s(b\cd b')\cd y \cd (y \ra x))\\
&=s(b\cd b')\ra (s(b\cd b')\cd x \cd (x \ra y))\\
&=f_{b\cd b'}(x \cd (x \ra y)),
\end{align*}
\end{comment}
\begin{align*}
f_{b\cd b'}(x \cd g_{b}(x \ra y))&
=s(b\cd b')\ra (s(b\cd b')\cd (x \cd (s(b)\ra(x \ra y))))\\
&=s(b\cd b') \ra (s(b) \cd s(b') \cd x \cd (s(b) \ra (x \ra y)))\\
&=s(b \cd b') \ra (s(b) \cd (s(b) \ra (x \ra y))  \cd s(b')\cd x)\\
&\overset{\text{H3}}{=}s(b \cd b') \ra ((x \ra y)\cd ((x \ra y) \ra s(b))\cd s(b') \cd x)\\
&=s(b \cd b') \ra ((x \ra y)\cd (sp(x \ra y) \ra s(b))\cd s(b') \cd x)
\\
&=s(b \cd b') \ra ((x \ra y)\cd (1 \ra s(b))\cd s(b') \cd x)\\
&=s(b \cd b') \ra ((x \ra y)\cd s(b) \cd s(b') \cd x)\\
&=s(b \cd b') \ra ( s(b \cd b') \cd x\cd (x \ra y))
\\&=f_{b \cd b'}(x \cd (x \ra y))
\end{align*}
for any $x,y \in X$ and $b,b' \in B$.
\item[E4.] It is possible to check that
\begin{align*}
g&_{b_{3}\ra (b_{1}\cd b_{2})}(f_{b_{1}\cd b_{2}}(x\cd y)\ra f_{b_{3}}(z))\\
&=s(b_{3}\ra (b_{1}\cd b_{2}))\ra ((s(b_{1}\cd b_{2})\ra (s(b_{1}\cd b_{2})\cd x\cd y))\ra (s(b_3)\ra s(b_3)\cd  z))\\
&\overset{\text{H4}}{=}s(b_3 \ra (b_1\cd b_2))\ra ((s(b_3) \cd ((s(b_1)\cd s(b_2))\ra (s(b_1)\cd s(b_2)\cd x\cd y)))\ra (s(b_3)\cd z))\\
&\overset{\text{H4}}{=}(s(b_3) \cd (s(b_1)\cd s(b_2)) \ra (s(b_1)\cd  s(b_2)\cd x\cd y)) \cd (s(b_3) \ra (s(b_1)\cd s(b_2))))\ra (s(b_3)\cd z)\\
&\overset{\text{H3}}{=}(((s(b_1)\cd s(b_2)) \ra (s(b_1)\cd s(b_2)\cd x\cd y)) \cd s(b_1)\cd s(b_2) \cd ((s(b_1)\cd s( b_2)) \ra s(b_3)))\ra (s(b_3)\cd z)\\
&\overset{\text{H3}}{=}(s(b_1)\cd s(b_2)\cd x\cd y \cd ((s(b_1)\cd s(b_2)\cd x\cd y) \ra (s(b_1)\cd s(b_2)))\cd ((s(b_1)\cd s(b_2)) \ra s(b_3)))\ra (s(b_3)\cd z)\\
&=(s(b_1)\cd s(b_2)\cd x\cd y \cd ((s(b_1)\cd s(b_2)) \ra s(b_3)))\ra (s(b_3)\cd z)\\
&\overset{\text{H3}}{=}(x\cd y\cd s(b_3)\cd (s(b_3) \ra (s(b_1)\cd s(b_2))))\ra (s(b_3)\cd z)\\
&\overset{\text{H4}}{=}(x\cd y \cd(s(b_3) \ra (s(b_1)\cd s(b_2))))\ra (s(b_3) \ra (s(b_3)\cd z)),
\end{align*} 
and
\begin{comment}
\begin{align*}
g&_{(b_{2}\ra b_{3})\ra b_{1}}(x\ra g_{b_{3}\ra b_{2}}(y \ra z))\\ 
&=((b_{2}\ra b_{3})\ra b_{1}) \ra (x \ra ((b_{3}\ra b_{2})\ra (y \ra z)))\\
&=((b_{2}\ra b_{3})\ra b_{1})\ra (x \ra ((y \cd(b_{3}\ra b_{2})\ra (b_{3}\ra (b_{3}\cd z)))))\\&
=((b_{2}\ra b_{3})\ra b_{1})\ra (x \ra (y \ra ((b_{3}\ra b_{2})\ra (b_{3}\ra (b_{3}\cd z))))\\
&=((b_{2}\ra b_{3})\ra b_{1})\ra (x \ra (y \ra (b_{2}\cd (b_{2}\ra b_{3})\ra (b_{3}\cd z)))\\
&=((b_{2}\ra b_{3})\ra b_{1})\ra (x \ra ((b_{2}\cd(b_{2}\ra (b_{2}\cd y))\cd (b_{2}\ra b_{3}))\ra (b_{3}\cd z))\\
&=((b_{2}\ra b_{3})\ra b_{1})\ra ((x\cd b_{2}\cd y\cd ((b_{2}\cd y)\ra b_{2})\cd(b_{2}\ra b_{3}))\ra (b_{3}\cd z))\\
&=((b_{2}\ra b_{3})\cd((b_{2}\ra b_{3})\ra b_{1})\cd x\cd y\cd b_{2}\cd ((b_{2}\cd y) \ra b_{2}))\ra (b_{3}\cd z)\\
&=(b_{1}\cd ((b_{1}\cd b_{2}) \ra b_{3})\cd x\cd y\cd b_{2}\cd ((b_{2}\cd y)\ra b_{2}))\ra (b_{3}\cd z)\\
&=(b_{1}\cd b_{2}\cd ((b_{1}\cd b_{2})\ra b_{3})\cd y\cd ((b_{2}\cd y)\ra b_{2})\cd (b_{1}\ra (b_{1}\cd x)))\ra (b_{3}\cd z)\\
&=(b_{1}\cd b_{2}\cd x\cd y\cd ((b_{1}\cd x) \ra b_{1})\cd((b_{2}\cd y) \ra b_{2}))\cd ((b_{1}\cd b_{2})\ra b_{3}))\ra (b_{3}\cd z)\\
&=(b_{3}\cd (b_{3}\ra (b_{1}\cd b_{2}))\cd x\cd ((b_{1}\cd x)\ra b_{1})\cd y\cd ((b_{2}\cd y)\ra b_{2}))\ra (b_{3}\cd z)\\
&=((b_{3}\ra (b_{1}\cd b_{2}))\cd x\cd ((b_{1}\cd x)\ra b_{1})\cd y\cd ((b_{2}\cd y)\ra b_{2}))\ra z\\
&=((b_{3}\ra (b_{1}\cd b_{2}))\cd x\cd y)\ra z,
\end{align*}
\end{comment}
\begin{align*}
g&_{(b_{2}\ra b_{3})\ra b_{1}}(x\ra g_{b_{3}\ra b_{2}}(y \ra f_{b_3}(z)))\\ 
&=s((b_{2}\ra b_{3})\ra b_{1}) \ra (x \ra (s(b_{3}\ra b_{2})\ra (y \ra (s(b_3) \ra (s(b_3)\cd z)))))\\
&\overset{\text{H4}}{=}s((b_{2}\ra b_{3})\ra b_{1})\ra (x \ra ((y \cd s(b_{3}\ra b_{2}))\ra (s(b_{3})\ra (s(b_{3})\cd z))))\\
&\overset{\text{H4}}{=}s((b_{2}\ra b_{3})\ra b_{1})\ra (x \ra (y \ra (s(b_{3}\ra b_{2})\ra (s(b_{3})\ra (s(b_{3})\cd z))))) \\
&\overset{\text{H4}}{=}s((b_{2}\ra b_{3})\ra b_{1})\ra (x \ra (y \ra (s(b_{3})\cd (s(b_{3})\ra s(b_{2}))\ra (s(b_{3})\cd z))))\\
&\overset{\text{H3}}{=}s((b_{2}\ra b_{3})\ra b_{1})\ra (x \ra (y \ra (s(b_{2})\cd (s(b_{2})\ra s(b_{3}))\ra (s(b_{3})\cd z))))\\
&\overset{\text{H4}}{=}s((b_2 \ra b_3)\ra b_1)\ra (x \ra ((y \cd (s(b_2) \cd (s(b_2) \ra s(b_3))))\ra (s(b_3)\cd z)))\\
&\overset{\text{H4}}{=}s((b_2 \ra b_3)\ra b_1)\ra ((x\cd y\cd s(b_2)\cd (s(b_2) \ra s(b_3)))\ra (s(b_3)\cd z))\\
&\overset{\text{H4}}{=}(((s(b_2) \ra s(b_3))\ra s(b_1))\cd (x \cd y \cd s(b_2) \cd (s(b_2) \ra s(b_3))))\ra (s(b_3)\cd z)\\
&=((s(b_2) \ra s(b_3))\cd ((s(b_2) \ra s(b_3)) \ra s(b_1))\cd x \cd y \cd s(b_2))\ra (s(b_3)\cd z)\\
&\overset{\text{H3}}{=}(s(b_1) \cd (s(b_1) \ra (s(b_2) \ra s(b_3)))\cd x \cd y \cd s(b_2))\ra (s(b_3)\cd z)\\
&=(s(b_1) \cd s(b_2) \cd ((s(b_1)\cd s(b_2)) \ra s(b_3))\cd x \cd y)\ra (s(b_3)\cd z)\\
&\overset{\text{H3}}{=}(s(b_3) \cd (s(b_3) \ra (s(b_1)\cd s(b_2)))\cd x \cd y)\ra (s(b_3)\cd z)\\
&\overset{\text{H4}}{=}((s(b_3) \ra (s(b_1)\cd s(b_2)))\cd x \cd y)\ra (s(b_3) \ra (s(b_3)\cd z)),
\end{align*}
for any $x,y,z \in X$ and $b_1,b_2,b_3 \in B$, where in equalities above we used the identity
\[
(x \cd s(b))\ra s(b)=sp(x \cd s(b))\ra s(b)=(1 \cd s(b))\ra s(b)=1.
\]
\end{itemize}
\end{proof}

As a consequence, we get the following.

\begin{proposition}\label{prop}
Let $B,X$ be hoops. There is a bijection $\tau_B$ between the sets $\spl(B,X)$ and $\EAct(B,X)$.
\end{proposition}

\begin{proof}
We define $\tau_B$ as follows: given any split extension with strong section in the variety $\Hp$ 
\[
\begin{tikzcd}
X & A & B,
\arrow["k", from=1-1, to=1-2]
\arrow["p", shift left, from=1-2, to=1-3]
\arrow["s", shift left, from=1-3, to=1-2]
\end{tikzcd}
\] 
we associate the strong external action $f,g \colon B \times X \rightarrow X$ defined by 
\[
f_{b}(x)=s(b)\ra (s(b)\cdot x) \; \text{ and } \; g_{b}(x)=s(b)\ra x.
\]
One may check that $\tau_B$ is a bijection whose inverse is the map $\mu_B$ which sends any strong external action $f,g \colon B \times X \rightarrow X$ to the split extension \eqref{splext} of \Cref{strong_to_split}.
\end{proof}

\begin{example}
Let $B$ and $X$ be hoops, and consider their direct product $A=X \times B$. The strong external action associated with the split extension described in \Cref{ex_direct} is given by the pair of maps \begin{equation*}
f,g \colon B \times X \rightarrow X
\end{equation*}
defined by
\[
f_b(x) = g_b(x) = x,
\]
for any $b \in B$ and $x \in X$.
\end{example}

\begin{example}
Let $A$ be a BL-algebra. The strong external action associated with the split extension of \Cref{ex_Rump} is given by the pair of maps
\begin{equation*}\label{strong_ex}
f,g \colon \operatorname{MV}(A) \times \operatorname{D}(A) \rightarrow \operatorname{D}(A)
\end{equation*}
defined by
\[
f_b(x)=b \ra (b \cd x) \; \text{ and } \; g_b(x)=b \ra x,
\]
for any $b \in \operatorname{MV}(A)$ and $x \in \operatorname{D}(A)$.
\end{example}

\begin{remark}\label{operates}
Let $f,g \colon B \times X \rightarrow X$ be a strong external action of hoops. It is possible to check that the map $g$ satisfies the following additional properties:
\begin{itemize}
\item[(i)] $g_b$ is monotone for any $b \in B$, i.e., if $x \leq y$, then $g_{b}(x)\leq g_{b}(y)$;
\item[(ii)] $g_{b \cdot b'}=g_b \circ g_{b'}$, for any $b,b' \in B$;
\item[(iii)] $g_{b}(x \ra y)=g_b(x) \ra g_{b}(y)$, for any $b \in B$ and $x,y \in X$.
\end{itemize}
These directly follow from the fact that the set
\[
Y'=\{(x,b) \in X \times B \mid f_{b}(x)=x\}
\]
endowed with the binary operations \eqref{op1} and \eqref{op2} is a hoop. %Moreover, it follows from the equality
%\[
%(1,b)\ra (x \ra y,1)=((1,b)\ra (x,1))\ra ((1,b)\ra (y,1))
%\]
%that
%\[
%(g_{b}(x\ra y),1)=(g_{b}(x)\ra g_{b}(y),1).
%\]
\end{remark}

Actually, it is possible to prove that the bijection $\tau_B$ of \Cref{prop} extends to a natural isomorphism of functors.

\begin{theorem}
There is a natural isomorphism
\[
\tau \colon \spl(-,X) \cong \EAct(-,X).
\]
\end{theorem}

\begin{proof}
The bijection $\tau_B$ of \Cref{prop} is natural in $B$, that is, for any morphism $\varphi \colon B' \ra B$ in $\Hp$ the diagram 
\[
\begin{tikzcd}
\spl(B,X) & \EAct(B,X) \\
\spl(B',X) & \EAct(B',X)
\arrow["\tau_B", from=1-1, to=1-2]
\arrow["{\varphi^{*}}"', from=1-1, to=2-1]
\arrow["{\EAct(\varphi,X)}", from=1-2, to=2-2]
\arrow["\tau_{B'}", from=2-1, to=2-2]
\end{tikzcd}
\]
commutes. Indeed, both the compositions $\EAct(\varphi,X) \circ \tau_B$ and $\tau_{B'} \circ \varphi^{*}$ sends a split extension with strong section
\[
\label{splext2}
\begin{tikzcd}
X & Y' & B
\arrow["{\iota_{1}}", from=1-1, to=1-2]
\arrow["{\pi_{2}}", shift left, from=1-2, to=1-3]
\arrow["{\iota_{2}}", shift left, from=1-3, to=1-2]
\end{tikzcd}
\]
to the strong external action $f',g' \colon B' \times X \ra X$ defined by
\[
f'(b',x)=f(\varphi(b'),x) \; \text{ and } g'(b',x)=g(\varphi(b'),x),
\]
for any $b' \in B'$ and $x \in X$.
\end{proof}

The results obtained above can be specialized to the subvarieties of basic hoops, Wajsberg hoops,  Gödel hoops and product hoops. Here we do not provide a detailed description for product hoops, since strong external actions in the variety $\PH$ were already described in~\cite{strongsect}.

\subsection{Strong external actions of basic hoops}

We start by showing how the lattice operations of \Cref{lattice_strongsect} in the semidirect product $X\rtimes_{\xi} B$ may be described in terms of the external action associated with $\xi$.

\begin{proposition}
Let $B,X$ be basic hoops, let $f,g \colon B \times X \rightarrow X$ be a strong external action and let
\begin{equation*}
\begin{tikzcd}
X & Y' & B
\arrow["{\iota_{1}}", from=1-1, to=1-2]
\arrow["{\pi_{2}}", shift left, from=1-2, to=1-3]
\arrow["{\iota_{2}}", shift left, from=1-3, to=1-2]
\end{tikzcd}
\end{equation*}
be the corresponding split extension. Then the lattice operations of join and meet in $Y'$ can be described as follows:
\[
(x,b)\wedge (y,b)=(f_{b \wedge b'}(x \wedge y), \ b \wedge b')
\]
and
\[
(x,b)\vee (y,b')=(f_{b \vee b'}((g_{b'\ra b}(x \ra y) \ra y)\wedge (g_{b \ra b'}(y \ra x)\ra x)), \ b \vee b'),
\]
for any $(x,b),(y,b') \in Y'$.
\end{proposition}

\begin{proof}
Straightforward computations show that
\begin{align*}
(x,b)\wedge (y,b')&=(x,b)\cd ((x,b)\ra (y,b'))\\
&=(x,b)\cd (g_{b' \ra b}(x \ra y), \ b \ra b')\\
&=(f_{b' \cd (b' \ra b)}(x \cd g_{b' \ra b}(x \ra y)), \ b \cd ( b \ra b'))\\
&\overset{\text{E3}}{=}(f_{b' \cd (b' \ra b)}(x \cd (x \ra y)), \ b \wedge b')\\
&=(f_{b \wedge b'}(x \wedge y), \ b \wedge b')
\end{align*}
and
\begin{align*}
(((x, b)\ra &(y, b'))\ra (y,b'))\wedge (((y,b')\ra (x,b))\ra (x,b))\\
&=((g_{b' \ra b}(x \ra y), \ b \ra b')\ra (y,b'))\wedge ((g_{b \ra b'}(y \ra x), \ b' \ra b) \ra (x,b))\\
&=(g_{b' \ra (b \ra b')}(g_{b' \ra b}(x \ra y)\ra y), \ (b \ra b')\ra b') \\ &\wedge (g_{b \ra (b' \ra b)}(g_{b \ra b'}(y \ra x)\ra x), \ (b' \ra b)\ra b)\\
&=(g_{1}(g_{b' \ra b}(x \ra y)\ra y), \ (b \ra b')\ra b') \\ &\wedge (g_{1}(g_{b \ra b'}(y \ra x)\ra x), \ (b' \ra b)\ra b)\\
&\overset{\text{E2}}{=}(g_{b' \ra b}(x \ra y) \ra y, \ (b \ra b')\ra b') \wedge (g_{b \ra b'}(y \ra x)\ra x, \ (b' \ra b) \ra b)\\
&=(f_{((b \ra b')\ra b')\wedge ((b' \ra b)\ra b)}((g_{b' \ra b}(x \ra y) \ra y)\wedge (g_{b \ra b'}(y \ra x)\ra x)),\\ & \ ((b \ra b')\ra b') \wedge ((b' \ra b) \ra b)\\
&=(f_{b \vee b'}((g_{b' \ra b}(x \ra y) \ra y) \wedge (g_{b \ra b'}(y \ra x) \ra x)), \ b \vee b'),
\end{align*}
for any $(x,b),(y,b') \in Y$, where we used that $b \ra (b' \ra b)=b' \ra (b \ra b')=1$. %since
%\[
%b \ra (b' \ra b)=b' \ra (b \ra b)=b' \ra 1=1
%\]
%and similarly
%\[
%b' \ra (b \ra b')=1.
%\]
\end{proof}

\begin{definition}\label{def_basic}
Let $B,X$ be basic hoops. A \emph{strong external action} of $B$ on $X$ in the variety $\BH$ consists of a pair of maps
\begin{align*}
&f \colon B \times X \rightarrow X \colon (b,x) \mapsto f_b(x),\\
&g \colon B \times X \rightarrow X \colon (b,x) \mapsto g_b(x)
\end{align*}
such that
\begin{enumerate}
\item[B1.] $(f,g)$ is a strong external action of hoops;
\item[B2.] For any $b_1,b_2,b_3 \in B$ and $x,y,z \in X$
\begin{align*}
g_{\tilde{b}}(g_{b_{3}\ra (b_1 \ra b_2)}(g_{b_2 \ra b_1}&(x \ra y)\ra z) \\ \ra (g&_{b_{3}\ra (b_{2}\ra b_{1})}(g_{b_{1}\ra b_{2}}(y \ra x)\ra z)\ra z))=1, 
\end{align*}
where
\[
\tilde{b}=(((b_{2}\ra b_{1})\ra b_{3})\ra b_{3})\ra ((b_{1}\ra b_{2})\ra b_{3}).
\]
\end{enumerate}
\end{definition}

Given any basic hoop $X$, consider the functor
\[
\operatorname{EAct_{ss}^B}(-,X) \colon \BH^{\op}\ra \ensuremath{\mathbf{Set}}
\]
which maps every basic hoop $B$ to the set $\operatorname{EAct_{ss}^{B}}(B,X)$ of strong external actions of $B$ on $X$ in the variety $\BH$ and
\[
\splb(-,X) \colon  \BH^{\op} \ra \ensuremath{\mathbf{Set}}
\]
which assigns to any basic hoop $B$, the set $\splb(B,X)$ of isomorphism classes of split extensions with strong section of $B$ by $X$ in $\BH$. As in the case of hoops, one may prove the following.

\begin{proposition}\label{prop2}
Let $B,X$ be basic hoops. There is a bijection $\tau_B$ between $\splb(B,X)$ and $\operatorname{EAct_{ss}^B}(B,X)$.
\end{proposition}

\begin{proof}
Let $\tau_B$ be the function that maps any split extension with strong section
\[\begin{tikzcd}
X & A & B,
\arrow["k", from=1-1, to=1-2]
\arrow["p", shift left, from=1-2, to=1-3]
\arrow["s", shift left, from=1-3, to=1-2]
\end{tikzcd}\] 
in the variety $\BH$ to the pair of maps $f,g \colon B \times X \rightarrow X$ defined by 
\[
f_{b}(x)=s(b)\ra (s(b)\cdot x) \; \text{ and } \; g_{b}(x)=s(b)\ra x.
\]
It is possible to verify that the pair $(f,g)$ defines a strong external action of $B$ on~$X$ in the variety of basic hoops.

Now, consider the map $\mu_B$ which sends a strong external action $f,g \colon B \times X \rightarrow X$ in $\BH$ to the split extension \eqref{splext} of \Cref{strong_to_split}. We already kwow that $Y'$ is a hoop and that the split extension \eqref{splext} has strong section. It remains to show that $Y'$ is a basic hoop. In fact, we have
\begin{align*}
(((x,b_1)&\ra (y,b_2))\ra (z,b_3))\ra ((((y,b_2)\ra (x,b_1))\ra (z,b_3))\ra (z,b_3))\\
&=((g_{b_{2}\ra b_{1}}(x \ra y), b_{1}\ra b_{2})\ra (z,b_3))\ra (((g_{b_{1}\ra b_{2}}(y \ra x), \ b_{2}\ra b_{1}) \ra (z,b_{3}))\ra (z,{b_{3}}))\\
&=(g_{b_{3}\ra (b_1 \ra b_2)}(g_{b_2 \ra b_1}(x \ra y)\ra z),(b_1 \ra b_2)\ra b_3)\\&\ra ((g_{b_{3}\ra (b_{2}\ra b_{1})}(g_{b_{1}\ra b_{2}}(y \ra x)\ra z), \ (b_{2}\ra b_{1})\ra b_{3})\ra (z, b_{3}))\\
&=(g_{b_{3}\ra (b_1 \ra b_2)}(g_{b_2 \ra b_1}(x \ra y)\ra z), \ (b_1 \ra b_2)\ra b_3)\\&\ra
(g_{b_{3}\ra ((b_{2}\ra b_{1})\ra b_{3})}(g_{b_{3}\ra (b_{2}\ra b_{1})}(g_{b_{1}\ra b_{2}}(y \ra x)\ra z)\ra z), \ ((b_{2}\ra b_{1})\ra b_{3})\ra b_{3})\\
&=(g_{(((b_{2}\ra b_{1})\ra b_{3})\ra b_{3})\ra ((b_{1}\ra b_{2})\ra b_{3})}(g_{b_{3}\ra (b_1 \ra b_2)}(g_{b_2 \ra b_1}(x \ra y)\ra z)\\ &\ra g_{b_{3}\ra ((b_{2}\ra b_{1})\ra b_{3})}(g_{b_{3}\ra (b_{2}\ra b_{1})}(g_{b_{1}\ra b_{2}}(y \ra x)\ra z)\ra z)),\\
&((b_{1}\ra b_{2})\ra b_{3})\ra (((b_{2}\ra b_{1})\ra b_{3})\ra b_{3}))\\
&\overset{\text{E2}}{=}(g_{\tilde{b}}(g_{b_{3}\ra (b_1 \ra b_2)}(g_{b_2 \ra b_1}(x \ra y)\ra z)\\ &\ra (g_{b_{3}\ra (b_{2}\ra b_{1})}(g_{b_{1}\ra b_{2}}(y \ra x)\ra z)\ra z)), \ 1)\overset{\text{B2}}{=}(1,1),
%&((b_{1}\ra b_{2})\ra b_{3})\ra (((b_{2}\ra b_{1})\ra b_{3})\ra b_{3})=(1,1).
\end{align*}
for any $(x,b_1),(y,b_2),(z,b_3) \in Y'$, where $\tilde{b}$ is as in \Cref{def_basic}, and we used the identity
\[
b_{3}\ra ((b_{2}\ra b_{1})\ra b_{3})=(b_{2}\ra b_{1})\ra (b_{3}\ra b_{3})=1
\]
and the fact that $B$ is a basic hoop. To conclude the proof, one may check that $\mu_B$ is the inverse of the map~$\tau_B$.
\end{proof}

Since, again, the bijection $\tau_B$ is natural in $B$, we get the following.

\begin{theorem}
There is a natural isomorphism
\[
\tau \colon \splb(-,X) \cong \operatorname{EAct_{ss}^B}(-,X). 
\]
\noproof
\end{theorem}

\subsection{Strong external actions of Wajsberg hoops}

As in the case of basic hoops, strong external actions in the variety $\WH$ can be defined as external actions of hoops satisfying an additional identity.

\begin{definition}
Let $B,X$ be Wajsberg hoops. A \emph{strong external action} of $B$ on~$X$ in the variety $\WH$ consists of a pair of maps
\begin{align*}
&f \colon B \times X \rightarrow X \colon (b,x) \mapsto f_b(x),\\
&g \colon B \times X \rightarrow X \colon (b,x) \mapsto g_b(x)
\end{align*}
such that
\begin{itemize}
\item[W1.] $(f,g)$ is a strong external action of hoops;
\item[W2.] $g_{b_{2}\ra b_{1}}(x \ra y)\ra y=g_{b_{1}\ra b_{2}}(y \ra x)\ra x$, for any $b_{1},b_{2} \in B$ and $x,y \in X$.
\end{itemize}
\end{definition}

\begin{remark}
Altough it is not required in the definition, one may check that the pair $(f,g)$ defines a strong external action in the variety of basic hoops.
\end{remark}

As in the previous sections, we define the functors 
\[
\operatorname{EAct_{ss}^W}(-,X) \colon \WH^{\op}\ra \ensuremath{\mathbf{Set}}
\]
which maps every Wajsberg hoop $B$ to the set $\operatorname{EAct_{ss}^W}(B,X)$ of strong external actions of $B$ on $X$, and
\[
\splw(-,X) \colon  \WH^{\op} \ra \ensuremath{\mathbf{Set}}
\]
which sends $B$ to the set $\splw(B,X)$ of isomorphism classes of split extensions with strong section of $B$ by $X$ in $\WH$.

The results of \Cref{prop} and \Cref{prop2} extend to the variety $\WH$.

\begin{proposition}
Let $B,X$ be Wajsberg hoops. There is a bijection $\tau_B$ between $\splw(B,X)$ and $\operatorname{EAct_{ss}^W}(B,X)$.
\end{proposition}

\begin{proof}
Let $\tau_B$ be the function that maps any split extensions with strong section
\[\begin{tikzcd}
X & A & B,
\arrow["k", from=1-1, to=1-2]
\arrow["p", shift left, from=1-2, to=1-3]
\arrow["s", shift left, from=1-3, to=1-2]
\end{tikzcd}
\] 
in $\WH$ to the strong external action $f,g \colon B \times X \rightarrow X$ defined by 
\[
f_{b}(x)=s(b)\ra (s(b)\cdot x) \; \text{ and } \; g_{b}(x)=s(b)\ra x.
\]
One may verify that $\tau_B$ is a bijection whose inverse is the map $\mu_B$ which sends any strong external action $f,g \colon B \times X \rightarrow X$ in $\WH$ to the split extension \eqref{splext} of \Cref{strong_to_split}. In fact, we already know that $Y'$ a hoop and that the split extension \eqref{splext} has strong section. Furthermore, $Y'$ is a Wajsberg hoop since
%i.e., it satisfies the identity
%\[
%((x,b)\ra (y,b'))\ra (y,b')=((y, b')\ra (x,b))\ra %(x,b),
%\]
%for any $(x,b),(y,b') \in Y$. 
\begin{align*}
((x,b)\ra (y,b'))\ra (y,b')&=(g_{b' \ra b}(x \ra y), \ b \ra b')\ra (y,b')\\
&=(g_{b' \ra (b \ra b')}(g_{b' \ra b}(x \ra y)\ra y), \ (b \ra b')\ra b')\\
&=(g_{1}(g_{b' \ra b}(x \ra y)\ra y), \ (b \ra b')\ra b')\\
&\overset{\text{E2}}{=}(g_{b' \ra b}(x \ra y)\ra y, \ (b \ra b')\ra b')\\
&\overset{\text{W2}}{=}(g_{b \ra b'}(y \ra x)\ra x, \ (b' \ra b) \ra b)\\
&=((y,b')\ra (x,b))\ra (x,b),
\end{align*}
for any $(x,b),(y,b') \in Y'$. 
\end{proof}

\begin{theorem}
The bijection $\tau_B$ extends to a natural isomorphism
\[
\tau \colon \splw(-,X) \cong \operatorname{EAct_{ss}^W}(-,X).
\] \noproof
\end{theorem}

The following remark shows that the notion of split extension with strong section trivializes in the case of bounded Wajsberg hoops.
 
\begin{remark}\label{rem_MV}
Let 
\[
\begin{tikzcd}
X & A & B
\arrow["k", from=1-1, to=1-2]
\arrow["p", shift left, from=1-2, to=1-3]
\arrow["s", shift left, from=1-3, to=1-2]
\end{tikzcd}
\] 
be a split extension with strong section in the variety $\WH$, and suppose that $s$ is a morphism of bounded hoops. In particular, $s(0)$ is the bottom element of $A$. Then $p$ is an isomorphism with inverse $s$. 

Indeed, for any $a \in A$, the strong section condition
\[
a \ra s(0)=sp(a) \ra s(0),
\]
implies
\[
(a \ra s(0))\ra s(0)=(sp(a)\ra s(0))\ra s(0).
\]
Since $A$ is a Wajsberg hoop, the equality above can be written as
\[
(s(0) \ra a) \ra a=(s(0) \ra sp(a)) \ra sp(a)
\]
Finally, since $s(0)$ is the bottom element of $A$, we have $s(0) \ra a = 1$ and $s(0) \ra s(p(a)) = 1$, so the previous equality reduces to
\[
a = s(p(a)).
\]
Hence $s \circ p = \id_A$. Since $p \circ s = \id_B$, it follows that $p$ is an isomorphism with inverse $s$.
\end{remark}

\subsection{Strong external actions of Gödel hoops} 

Let $B$ and $X$ be Gödel hoops. In this section, we prove that any strong external action of $B$ on $X$ in the variety of basic hoops gives rise to an action in the variety $\GH$. 

\begin{proposition}
Let $B$ and $X$ be Gödel hoops and let 
\begin{equation}\label{eq_splt_godel}
\begin{tikzcd}
X & A & B,
\arrow["k", from=1-1, to=1-2]
\arrow["p", shift left, from=1-2, to=1-3]
\arrow["s", shift left, from=1-3, to=1-2]
\end{tikzcd}   
\end{equation}
be a split extension with strong section in the variety of basic hoops. Then $A$ is a Gödel hoop.
\end{proposition}

\begin{proof}
Let $(f,g)$ be the strong external action in the variety $\BH$ associated with \eqref{eq_splt_godel}, and consider the split extension with strong section
\begin{equation*}
\begin{tikzcd}
X & Y' & B
\arrow["{\iota_{1}}", from=1-1, to=1-2]
\arrow["{\pi_{2}}", shift left, from=1-2, to=1-3]
\arrow["{\iota_{2}}", shift left, from=1-3, to=1-2]
\end{tikzcd}
\end{equation*}
associated with $(f,g)$, where
\[
Y'=\{(x,b) \in X \times B \mid f_{b}(x)=x\}
\]
and the binary operations $\ra$ and $\cdot$ are defined as in \Cref{split_to_strong}. One can check that $Y'$ is a Gödel hoop, since
\begin{align*}
(x,b)\cd (x,b)&=(f_{b \cd b}(x \cd x), b \cd b)\\
&=(f_{b}(x),b)\\
&=(x,b),
\end{align*}
for any $(x,b) \in Y'$.
\end{proof}

This justifies the following definition.
    	
\begin{definition}
Let $B,X$ be Gödel hoops. A \emph{strong external actions} of $B$ on $X$ in the variety $\GH$ is a strong external actions of $B$ on $X$ in the variety of basic hoops.
\end{definition}

As a direct consequence, for any Gödel hoop $X$, the functors $\operatorname{EAct_{ss}^G}(-,X)$ and $\splg(-,X)$ can be defined as
\[
\operatorname{EAct_{ss}^G}(-,X)=\operatorname{EAct_{ss}^B}(U(-),U(X)) \; \text{ and } \; \splb = \operatorname{SplExt_{ss}^B}(U(-),U(X)),
\]
where $U \colon \GH \ra \BH$ denotes the forgetful functor, and we get the following.
 
\begin{theorem}
There is a natural isomorphism
\[
\tau \colon \splg(-,X) \cong \operatorname{EAct_{ss}^G}(-,X).
\] \noproof
\end{theorem}

\section{Comparison with W.~Rump's semidirect product}\label{sec_rump}

We end the article by showing a connection between the strong external actions in the variety of hoops and the non-categorical notion of semidirect product introduced by W.~Rump in the not protomodular context of L-algebras (see~\cite{Rump3, Rump2}).

\begin{definition}
An \emph{L-algebra} is an algebra $L=(L, \ra, 1)$ of type $(2,0)$ satisfying
\begin{itemize}
\item[L1.] $1 \ra x=x$ and $x \ra x=x \ra 1=1$, for any $x \in L$;
\item[L2.] $(x \ra y)\ra (x \ra z)=(y \ra x)\ra (y \ra z)$, for any $x,y,z \in L$;
\item[L3.] $x \ra y=y \ra x=1 \Rightarrow x =y$, for any $x,y \in L$.
\end{itemize}
\end{definition}

The binary operation $\ra$ can be interpreted as the logical implication, while the constant $1$ can be interpreted as the logical unit.

\begin{example}
Every hoop is an L-algebra.
\end{example}

\begin{definition}\label{def_op}
Let $B,X$ be L-algebras. We say that $B$ \emph{operates} on $X$ if there is a map 
\[
B \times X \ra X \colon (b,x) \mapsto bx
\]
such that
\begin{itemize}
\item[O1.] $b(x \ra y)=b x \ra b y$;
\item[O2.] $(b \ra b' )b x=(b' \ra b)b' x$;
\item[O3.] $1x=x$,
\end{itemize}
for any $b, b' \in B$ and $x,y \in X$.    
\end{definition}

We want to show how the notion of strong external action in the variety of hoops is related with \Cref{def_op}.

\begin{proposition}
Let $B,X$ be hoops and let $(f,g)$ be a strong external action of $B$ on $X$. Then $B$ operates on $X$ via $g$.    
\end{proposition}

\begin{proof}
It follows from \Cref{operates} that 
\[
g_{b}(x \ra y)=g_{b}(x)\ra g_{b}(y),
\]
for any $b \in B$ and $x,y \in X$.
Moreover, since $g_{b_{1}\cd b_{2}}=g_{b_{1}}\circ g_{b_{2}}$ for any $b,b' \in B$, we get
\begin{align*}
g_{b \ra b'}(g_{b}(x))&=g_{(b \ra b')\cd b}(x)\\
&\overset{\text{H3}}{=}g_{(b' \ra b)\cd b}(x)\\
&=g_{b' \ra b}(g_{b'}(x)).
\end{align*}
Furthermore, Axiom O3 follows from Axiom E2.    
\end{proof}

Since the category of L-algebras is not semi-abelian, it doesn't have categorical semidirect products. However, given two L-algebras $B$ and $X$ such that $B$ operates on $X$, we can still determine a corresponding notion of semidirect product.

\begin{definition}
Let $B,X$ be L-algebras such that $B$ operates on $X$. The \emph{semidirect product} $X \rtimes B$ of $B$ by $X$ is the cartesian product $X \times B$ endowed with the binary operation
\[
(x,b)\ra(y, b')=(((b \ra b')x)\ra (b' \ra b) y, \ b \ra b'),
\]
for any $b,b' \in B$ and $x,y \in X$, and with unit $1_{X \rtimes B}=(1,1)$.
\end{definition}

\begin{proposition}
Let $B,X$ be L-algebras such that $B$ operates on $X$. Then the semidirect product $X \rtimes B$ is an L-algebra. \noproof
\end{proposition}

\begin{definition}
We say that a split extension 
\[
\begin{tikzcd}
X & A & B
\arrow["k", from=1-1, to=1-2]
\arrow["p", shift left, from=1-2, to=1-3]
\arrow["s", shift left, from=1-3, to=1-2]
\end{tikzcd}\]
of L-algebras has a \emph{strong section}, if the equation 
\begin{equation*}
a \ra s(b)=sp(a)\ra s(b)
\end{equation*} holds for any $a \in A$ and $b \in B$.
\end{definition}

Given a split extension with strong section, it is possible to prove that $B$ operates on $X$. Indeed, the map
\[
B \times X \ra X \colon (b,x) \mapsto s(b)\ra x
\]
satisfies Axioms O1-O2-O3. Moreover, the semidirect product of $X$ and $B$ coincide with the cartesian product $X \times B$ endowed with the binary operation
\[
(x, b_{1})\ra (y, b_{2})=((s(b_{1}\ra b_{2})\ra x)\ra (s(b_{2}\ra b_{1})\ra y), \ b_{1}\ra b_{2})
\]
and with logical unit $(1,1)$.

\begin{remark}
One may check that the map
\[
A \ra X \times B \colon a \mapsto (sp(a)\ra a,p(a))
\]
is not an isomorphism of L-algebras, but just an injection. However, it is possible to prove that, if $A$ is a hoop, then the above map defines an isomorphism of L-algebras between $A$ and the subset
\[
Y'=\{(x,b)\in X \times B \mid s(b)\ra (s(b)\cd x)=x\}
\]
of $X \times B$
\end{remark}

\begin{proposition}
Let
\[
\begin{tikzcd}
X & A & B
\arrow["k", from=1-1, to=1-2]
\arrow["p", shift left, from=1-2, to=1-3]
\arrow["s", shift left, from=1-3, to=1-2]
\end{tikzcd}
\] 
be a split extensions with strong section in the variety of hoops. Then the operation
\[
(x, b)\ra (y, b')=((s(b\ra b')\ra x)\ra (s(b'\ra b)\ra y), \ b\ra b')
\]
defined on the set
\[
Y'=\{(x,b)\in X \times B \mid s(b)\ra (s(b)\cd x)=x\}
\]
coincides with the implication $\ra$ of \Cref{thss}.
\end{proposition}

\begin{proof}
It is sufficient to check that
\begin{align*}
&(s(b\ra b')\ra x)\ra(s(b' \ra b)\ra y)\\
&=s(s(b\ra b')\ra x)\ra (s(b'\ra b)\ra (s(b')\ra (s(b')\cd y)))\\
&\overset{\text{H4}}{=}(s(b\ra b')\ra x)\ra ((s(b')\cd s(b' \ra b))\ra (s(b')\cd y))\\
&\overset{\text{H3}}{=}(s(b \ra b')\ra x)\ra ((s(b)\cd s(b \ra b'))\ra (s(b')\cd y))\\
&\overset{\text{H4}}{=}((s(b)\cd s(b\ra b') \cd s(b \ra b'))\ra x)\ra (s(b')\cd y)\\
&\overset{\text{H3}}{=}(s(b)\cd x\cd (x \ra s(b \ra b')))\ra (s(b')\cd y)\\
&\overset{\text{H3}}{=}(s(b') \cd s(b' \ra b)\cd x)\ra (s(b')\cd y)\\
&\overset{\text{H4}}{=}(s(b' \ra b) \cd x) \ra (s(b')\ra (s(b')\cd y))\\
&=(s(b' \ra b) \cd x) \ra y\\
&\overset{\text{H4}}{=}s(b' \ra b) \ra (x \ra y),
\end{align*}
for any $(x,b),(y,b') \in Y$.
\end{proof}

\begin{remark}
If the hoop is \emph{self-similar}, then by cancellativity 
\[
s(b)\ra (s(b)\cd x)=x
\]
for any $(x,b) \in X \times B$. Hence, the domain of $X \rtimes B$ is $X \times B$ and $f_{b}=\id_X$ for any $b \in B$.
\end{remark}

\section{Conclusions}\label{sec_concl}

In this manuscript, we described split extensions with strong section in varieties of hoops in terms of \emph{strong external actions}. We proved that for any hoop $X$, there is a natural isomorphism
\[
\spl(-,X) \cong \EAct(-,X)
\]
showing that any split extension with strong section of $B$ by $X$ in the variety $\Hp$ can be described by a pair of maps
\[
f,g \colon B \times X \rightarrow X
\]
satisfying a set of identities which depends on the axioms satisfied by the hoop.

A natural direction for future research is to extend this correspondence to \emph{all} split extensions, with the aim of developing a general notion of external action and of establishing natural isomorphisms
\[
\Act(-,X) \cong \SplExt(-,X) \cong \operatorname{EAct}(-,X),
\]
thus relating internal and external actions, in analogy with the case of derived actions in Orzech categories of interest~\cite{Orzech}.
%Eventually, in \Cref{rem_MV} we showed that the notion of split extension with strong section becomes trivial in the case of MV-algebras. It remains an open problem to construct an explicit example of a (non-trivial) split extension with strong section in the variety of Wajsberg hoops.

\end{document}